\documentclass[reqno]{amsart}

\usepackage{mathrsfs}
\usepackage{amsmath}
\usepackage{amssymb}
\usepackage{cite}
\usepackage{latexsym}
\usepackage{graphicx}

\usepackage{color}
\usepackage{comment}

\theoremstyle{plain}
\newtheorem{thm}{Theorem}[section]

\newtheorem{cor}[thm]{Corollary}

\newtheorem{prop}[thm]{Proposition}
\newtheorem{rmk}[thm]{Remark}

\def\D{\mathrm{D}}
\def\K{\mathscr{K}}
\def\L{\mathscr{L}}
\def\P{\mathscr{P}}
\def\U{\mathscr{U}}

\def\c{\mathrm{c}}
\def\d{\mathrm{d}}

\def\h{\mathrm{h}}

\def\Cset{\mathbb{C}}
\def\Nset{\mathbb{N}}
\def\Pset{\mathbb{P}}
\def\Rset{\mathbb{R}}
\def\Sset{\mathbb{S}}

\def\id{\mathrm{id}}
\def\Re{\mathrm{Re}}

\def\epsilon{\varepsilon}

\makeatletter
 \@addtoreset{equation}{section}
\makeatother
\def\theequation{\arabic{section}.\arabic{equation}}

\begin{document}


\title[Kuramoto model with two-mode interaction]%
{Bifurcations of synchronized solutions in a continuum limit
 of the Kuramoto model with two-mode interaction depending on two graphs}
\thanks{The general content of the manuscript, in whole or in part,
 is not submitted, accepted, or published elsewhere, including conference proceedings.}

\author{Kazuyuki Yagasaki}

\address{Department of Applied Mathematics and Physics, Graduate School of Informatics,
Kyoto University, Yoshida-Honmachi, Sakyo-ku, Kyoto 606-8501, JAPAN}
\email{yagasaki@amp.i.kyoto-u.ac.jp}

\date{\today}
 
\begin{abstract}
We study bifurcations of the completely synchronized state
 in a continuum limit (CL) for the Kuramoto model (KM)
 of identical oscillators with two-mode interaction depending on two graphs.
Here one of the graphs is uniform
 but may be deterministic dense, random dense or random sparse,
 and the other is a deterministic finite nearest neighbor.
We use the center manifold reduction technique,
 which is a standard one in dynamical systems, and prove that the CL suffers bifurcations
 at which the one-parameter family of completely synchronized state becomes unstable
 and a stable two-parameter family of $\ell$-humped sinusoidal shape stationary solutions
 ($\ell\ge 2$) appears.
This contrasts the author's recent result on the classical KM
 for which bifurcation behavior in its CL is very different from ones in the KM
 and difficult to explain by standard techniques in dynamical systems
 such as the center manifold reduction.
Moreover, similar bifurcation behavior is shown to occur in the KM,
 based on the previous fundamental results.
The occurrence of such bifurcations were suggested by numerical simulations
 for the deterministic graphs in a previous study.
We also demonstrate our theoretical results
 by numerical simulations for the KM with the zero natural frequency.
\end{abstract}
\maketitle


\section{Introduction}

Coupled oscillators on complex networks, which are typically defined by graphs,
 have recently attracted much attention
 and have been extensively studied.
They provide many mathematical models in various fields
 ranging from physics, chemistry and biology to social sciences and engineering.
Synchronization is one of typical and interesting phenomena observed in them.
Among them,
 the Kuramoto model (KM) \cite{K75,K84} is one of the most representative ones
 and has been generalized in several directions.
They have been very frequently studied, especially  to describe the synchronization phenomenon. 
See \cite{S00,PRK01,ABVRS05,ADKMZ08,DB14,PR15,RPJK16}
 for the reviews of vast literature on coupled oscillators in complex networks
 including the KM and its generalizations.

Let $G_{kn}=\langle V(G_{kn}),E(G_{kn}),W(G_{kn})\rangle$, $n\in\Nset$,
 be weighted graphs for $k=1,2$,
 where $V(G_{kn})=[n]$ and $E(G_{kn})$ are the sets of nodes and edges, respectively,
 and $W(G_{kn})$ is an $n\times n$ weight matrix given by
\begin{equation*}
(W(G_{kn}))_{ij}=
\begin{cases}
w_{ij}^{kn} & \mbox{if $(i,j)\in E(G_{kn})$};\\
0 &\rm{otherwise}.
\end{cases}
\end{equation*}
We express
\[ 
E(G_{kn})=\{(i,j)\in[n]^2\mid (W(G_{kn}))_{ij}\neq 0\}, 
\] 
where each edge is represented by an ordered pair of nodes $(i,j)$, 
 which is also denoted by $j\to i$, and a loop is allowed.
If $W(G_{kn})$ is symmetric, 
 then $G_{kn}$ represents an undirected weighted graph 
 and each edge is also denoted by $i\sim j$ instead of $j\to i$. 
When $G_{kn}$ is a simple graph, 
 $W(G_{kn})$ is a matrix whose elements are $\{0,1\}$-valued. 
When $G_{kn}$ is a random graph, 
 $W(G_{kn})$ is a random matrix. 
We say that $G_{kn}$ is a \emph{dense} graph 
 if $\# E(G_{kn})/(\# V(G_{kn}))^2>0$ as $n \rightarrow \infty$. 
If $\# E(G_{kn})/(\# V(G_{kn}))^2\rightarrow 0$ as $n \rightarrow \infty$, 
 then we call it a \emph{sparse} graph.

In this paper we consider the following KM \cite{K75,K84} of identical oscillators
 with two-mode interaction depending on such two graphs $G_{kn}$, $k=1,2$:
\begin{align}
\frac{\d}{\d t} u_i^n (t)
=&\omega
+\frac{1}{n\alpha_{1n}}\sum_{j=1}^{n}w_{ij}^{1n}\sin \left( u_j^n(t) - u_i^n(t) \right) \notag\\
&
-\frac{K}{n \alpha_{2n}} \sum_{j=1}^{n}w_{ij}^{2n}\sin 2\left( u_j^n(t) - u_i^n(t) \right) ,\quad
i \in [n],
\label{eqn:dsys}
\end{align}
where $u_i^n \colon \Rset \rightarrow\Sset^1$
 stands for the phase of oscillator at the node $i \in [n]$,
 $\omega$ is the natural frequency,  $K>0$ is the coupling constant,
 and $\alpha_{kn}>0$ is a scaling factor that is one if $G_{kn}$ is dense
 and less than one with $\alpha_{kn}\searrow 0$ and $n \alpha_{kn} \to\infty$
 as $n\rightarrow \infty$, if $G_{kn}$ is sparse for each $k=1,2$.
For any constant $\theta\in\Sset^1$,
 $u_i^n(t)=\omega t+\theta$, $i\in[n]$, represent a completely synchronized state
 and are always a solution to the KM \eqref{eqn:dsys}, as shown easily.

Moreover, the weight matrix $W(G_{kn})$ is given for each $k=1,2$ as follows.
Let $I=[0,1]$ and let $W_k^n\in L^2 (I^2)$, $n\in\Nset$, be nonnegative functions.
If $G_{kn}$, $n\in\Nset$, are deterministic dense graphs, then
\begin{equation}
w_{ij}^{kn} = \langle W_k^n\rangle_{ij}^{n}
:= n^2 \int_{I_i^n \times I_j^n}W_k^n(x,y) \d x\d y,
\label{eqn:ddg}
\end{equation}
where
\[
I_i^n:=
\begin{cases}
  [(i-1)/n,i/n) & \mbox{for $i<n$};\\
  [(n-1)/n,1] & \mbox{for $i=n$}.
\end{cases}
\]
If $G_{kn}$, $n\in\Nset$, are random dense graphs,
 then $w_{ij}^{kn}=1$ with probability 
\begin{equation}
\Pset(j \rightarrow i) = \langle W_k^n\rangle_{ij}^{n}, 
\label{eqn:rdg}
\end{equation}
where the range of $W_k^n$ is contained in $I$. 
If $G_{kn}$, $n\in\Nset$, are random sparse graphs, 
then $w_{ij}^{kn}=1$ with probability 
\begin{equation}
\Pset(j \rightarrow i) = \alpha_{kn} \langle \tilde{W}_{k}^n \rangle_{ij}^n, 
\quad \tilde{W}_{k}^n(x,y) :=\alpha_{kn}^{-1} \wedge W_k^n(x,y), 
\label{eqn:rsg} 
\end{equation} 
where $\alpha_{kn} =n^{-\gamma_k}$ with $\gamma_k\in(0,\frac{1}{2})$,
 and $a\wedge b=\min(a,b)$ for $a,b\in\Rset$.
The function $W_k^n(x,y)$ is usually called a \emph{graphon} \cite{L12}.
Such a construction of a random graph where $W_k^n(x,y)$ does not depend on $n$
 was given in \cite{M19} and used in \cite{IY23}.
We assume that there exist measurable functions $W_k\in L^2(I^2)$, $k=1,2$, such that 
\begin{equation}
\|W_k(x,y)-W_k^n(x,y)\|_{L^2(I^2)}=\int_{I^2}|W_k(x,y)-W_k^n(x,y)|^2\d x\d y\to 0
\label{eqn:Wk}
\end{equation}
as $n\to\infty$.

We specifically choose  a uniform graph $W_1^n(x,y)=W_1(x,y)=p$
 which may be deterministic dense, random dense or random sparse
 in \eqref{eqn:dsys}, where $p$ is a constant, for the first graph, and
\[
W_2^n(x,y)=\begin{cases}
1 & \mbox{if $(x,y)\in I_i^n\times I_j^n$ with $|i-j|\le n\kappa$ or $|i-j|\ge n(1-\kappa)$};\\
0 & \mbox{otherwise},
\end{cases}
\]
and
\[
W_2(x,y)=\begin{cases}
1 & \mbox{if $|x-y|\le\kappa$ or $|x-y|\ge1-\kappa$};\\
0 & \mbox{otherwise},
\end{cases}
\]
with $\kappa\in(0,\tfrac{1}{2})$,
 i.e., a $\lfloor n\kappa\rfloor$-nearest neighbor graph,
 which is simply called a nearest neighbor graph hereafter, for the second graph,
 where $\lfloor z\rfloor$ represents the maximum integer that is not greater than $z\in\Rset$.

Such coupled oscillator networks defined on two graphs as \eqref{eqn:dsys}
 were studied from the viewpoint of control \cite{GC20,GC21,GCH23,GM22}
 although they are linear.
For instance, for the KM \eqref{eqn:dsys},
 the last term is added as a control input
 to avoid undesired synchronization occurring when $K=0$
 since it may yield traffic congestion in networks and their collapse.
The additional networks are desired to be smaller and simpler than the original ones,
 like a deterministic nearest neighbor graph
 in the KM \eqref{eqn:dsys},
 from a cost perspective,
 and it is natural to consider that the dynamics on them have different properties.
See also Section~5 of \cite{IY23}.

In \cite{IY23},
 coupled oscillator networks including \eqref{eqn:dsys} were studied
 and shown to be well approximated by the corresponding continuum limits (CLs),
 for instance, which are given by
\begin{align}
\frac{\partial}{\partial t}u(t,x)
=& \omega+p\int_I \sin(u(t,y)-u(t,x))\d y\notag\\
& -K\int_I W_2(x,y) \sin 2(u(t,y)-u(t,x))\d y, \quad x \in I,
\label{eqn:csys}
\end{align}
for \eqref{eqn:dsys} and $W_1(x,y)=p$.
More general cases in which the networks depend on more than two graphs
 or the natural frequency of each oscillator is different
 were discussed in \cite{IY23}.
Similar results for such networks which are defined on single graphs
 and do not have natural frequencies depending on each node
 were obtained earlier in \cite{KM17,M14a,M14b,M19}
 although they are not applicable to \eqref{eqn:dsys} and \eqref{eqn:csys}.
Such a CL was introduced for the classic KM,
 which depends on the single complete simple graph,
 without a rigorous mathematical guarantee very early in \cite{E85},
 and fully discussed very recently in \cite{Y24}.
Similar CLs were utilized
 for the KM with nonlocal coupling and a single or the zero natural frequency
 in \cite{GHM12,WSG06}.
For any constant $\theta\in\Sset^1$,
 the CL \eqref{eqn:csys} has the solution $u(t,x)=\omega t+\theta$,
 which corresponds to the completely synchronized state
 $u_i^n(t)=\omega t+\theta$, $i\in[n]$, in the KM \eqref{eqn:dsys}
 and which is stable when $K=0$ (see Proposition~\ref{prop:3b} below).
In an adequate rotational frame,
  we can take $\omega=0$ in both the KM \eqref{eqn:dsys} and CL \eqref{eqn:csys},
  and do so hereafter.
 
The KMs with two-mode interaction such as \eqref{eqn:dsys}
 were previously studied in \cite{CD99,CN11,IY23,KP13,KP14,KP15}.
In \cite{KP13,KP14}, the case in which the natural frequencies are randomly distributed
 with a symmetric one-hump density function
 but the graphs are deterministic and complete simple in our terminology
 was analyzed by the self-consistent approach \cite{K75,K84},
 and multiple synchronized states were obtained.
In \cite{CN11} such a situation was also analyzed
 by applying a center manifold reduction technique of \cite{C15},
 which provides a mathematical foundation
 for the the self-consistent approach \cite{K75,K84},
 to a continuous model which is different from the CL \eqref{eqn:csys},
 and a bifurcation diagram of synchronized states was obtained.
Moreover, an influence of noise on the synchronization
 was discussed in \cite{CD99,KP15},
 although their approaches were not guaranteed rigorously.
We remark that in these references
 the main interest is in how synchronized motions appear in the bi-harmonic interaction 
 and how different it is from the single harmonic interaction
 as well as the complete simple graphs were considered,
 although the natural frequencies are randomly distributed.

On the other hand, the case in which the graphs are deterministic,
 one of them is complete simple and  the other is a nearest neighbor
 although the KM has the zero natural frequency (equivalently,
 it consists of the identical oscillators like \eqref{eqn:dsys})
 was studied in \cite{IY23}.
The problem is also important from the viewpoint of applications
 since its investigation provides several hints and insights
 for avoiding traffic congestion in networks and their collapse
 like those in \cite{GC20,GC21,GCH23,GM22}, as stated above.
In such an application, it seems natural
 that coupled oscillators depend not on a single network but on two different networks.
The stability of the completely synchronized state
 was discussed by analyzing the CL \eqref{eqn:csys} with $p=1$,
 of which some relationships with the KM \eqref{eqn:dsys}
 were developed there and expanded subsequently in \cite{Y24}
 (see Section~2 for more details).
Moreover, different synchronized solutions resembling unstable modes
 were observed in numerical simulations
 when the completely synchronized state seems unstable.
  
Here we study bifurcations of the completely synchronized solution
 $u(t,x)\equiv\theta$ in the CL \eqref{eqn:csys},
 where $\theta\in\Sset^1$ is any constant
 and $\omega=0$ is taken through this paper as stated above,
 using the center manifold reduction technique \cite{HI11},
 which is a standard one in dynamical systems \cite{GH83,K04}.
To the author's knowledge,
 the technique has not been utilized to analyze bifurcations in CLs of type \eqref{eqn:csys}
 anywhere else.
In addition, no reference  except for \cite{IY23} has theoretically investigated the KM
 depending on multiple graphs.
It was also shown in \cite{Y24}
 that bifurcation behavior in the classical KM is very different from ones in its CL.
In particular, $O(2^n)$ saddle-node and pitchfork bifurcations occur in the former
 while no such bifurcation in the latter.
This indicates that standard techniques in dynamical systems
 such as the center manifold reduction may not work effectively
 to analyze  bifurcations in the CL.
However, we show that the standard approach still works well in our problem.
In particular, we prove that the CL \eqref{eqn:csys} suffers bifurcations
 at which the one-parameter family of completely synchronized state becomes unstable
 and a stable two-parameter family of $\ell$-humped  sinusoidal shape stationary solutions
 ($\ell\ge 2$) appears.
Moreover, it follows from the previous fundamental results of \cite{IY23,Y24}
 that similar bifurcation behavior occurs in the KM \eqref{eqn:dsys}
 (see Remark~\ref{rmk:4a}(iii) below).
The occurrence of such bifurcations were suggested in numerical simulations
 for the deterministic graphs mentioned above in \cite{IY23}.
We also demonstrate our theoretical results
 by numerical simulations for the KM \eqref{eqn:dsys} with $\omega=0$:
A complete simple (i.e., deterministic dense),
 random dense or sparse graph is chosen as $G_{1n}$
 and a deterministic dense graph as $G_{2n}$.

The outline of this paper is as follows:
In Section~2 we briefly review the previous fundamental results
 of \cite{IY23} and \cite{Y24} in the context of the KM \eqref{eqn:dsys} and CL \eqref{eqn:csys}
 for the reader convenience.
In the remaining sections,
 we treat the KM \eqref{eqn:dsys} and CL \eqref{eqn:csys} directly.
We analyze the associated linear eigenvalue problem and bifurcations of synchronized solutions
 for the CL \eqref{eqn:csys} in Sections~3 and 4, respectively.
Based on the fundamental result reviewed in Section~2,
 we see that bifurcation behavior detected in Section~4 for the CL \eqref{eqn:csys}
 also occurs in the KM \eqref{eqn:dsys}, as stated above.
Numerical simulation results for the KM \eqref{eqn:dsys} with $\omega=0$
 are given in Section~5.


\section{Previous Fundamental Results}

We first review the results of \cite{IY23,Y24}
 on relationships between couples oscillator networks and their CLs
 in the context of  \eqref{eqn:dsys} and \eqref{eqn:csys} with $\omega=0$.
See Section~2 and Appendices~A and B of \cite{IY23}
 and Section~2 of \cite{Y24} for more details
 including the proofs of the theorems stated below.
The theory can be extended to more general cases.

Let $g(x)\in L^2(I)$
 and let $\mathbf{u}:\Rset\to L^2(I)$ stand for an $L^2(I)$-valued function.
We have the following on the existence and  uniqueness of solutions
 to the initial value problem (IVP) of the CL \eqref{eqn:csys}
 (see Theorem~2.1 of \cite{IY23}).
 
\begin{thm}
\label{thm:2a}
There exists a unique solution $\mathbf{u}(t)\in C^1(\Rset,L^2(I))$
 to the IVP of \eqref{eqn:csys} with
\[
u(0,x)=g(x).
\]
Moreover, the solution depends continuously on $g$.
\end{thm}

We next consider the IVP of the KM \eqref{eqn:dsys}
 and turn to the issue on convergence of solutions in \eqref{eqn:dsys}
 to those in the CL \eqref{eqn:csys}.
Since the right-hand side of \eqref{eqn:dsys} is Lipschitz continuous in $u_i^n$, $i\in[n]$,
 we see by a fundamental result of ordinary differential equations
 (e.g., Theorem~2.1 of Chapter~1 of \cite{CL55})
 that the IVP of \eqref{eqn:dsys} has a unique solution.
Given a solution $u_n(t)=(u_1^n(t),\ldots, u_n^n(t))$ to the IVP of \eqref{eqn:dsys},
 we define an $L^2(I)$-valued function $\mathbf{u}_n:\Rset\to L^2(I)$ as
\begin{equation*}
\mathbf{u}_n(t) = \sum^{n}_{i=1} u_i^n(t) \mathbf{1}_{I_i^n},
\end{equation*}
where $\mathbf{1}_{I_i^n}$ represents the characteristic function of $I_i^n$, $i\in[n]$.
Let $\|\cdot\|$ denote the norm in $L^2(I)$.
In our setting as stated in Section~1,
 we slightly modify the arguments given in the proof of Theorem~2.3 of \cite{IY23}
 to obtain the following
 (see also Remark~2.3 of \cite{Y24}).

\begin{thm}
\label{thm:2b}
If $\mathbf{u}_{n}(t)$ is the solution to the IVP of \eqref{eqn:dsys}
 with the initial condition
\[
\lim_{n\to\infty}\|\mathbf{u}_n(0)-\mathbf{u}(0)\|=0\quad\mbox{a.s.},
\]
then for any $\tau > 0$ we have
\[
\lim_{n \rightarrow \infty}\max_{t\in[0,\tau]}\|\mathbf{u}_n(t)-\mathbf{u}(t)\|=0
\quad\mbox{a.s.},
\]
where $\mathbf{u}(t)$ represents the solution
 to the IVP of the CL \eqref{eqn:csys}.
\end{thm}

We consider the case in which $G_{1n}$, $n\in\Nset$, are deterministic dense and complete.
Let $\boldsymbol{\theta}$ represent the constant function $u=\theta$ for $\theta\in\Sset^1$,
 and let
\[
T_n^l\mathbf{u}_n(t)=\sum_{i=1}^{n-l}u_{i+l}^n(t)\mathbf{1}_i^n
 +\sum_{i=n-l+1}^{n}u_{i-n+l}^n(t)\mathbf{1}_i^n,\quad
l\in[n],
\]
and
\[
T^\psi\mathbf{u}(t)=\begin{cases}
u(t,x+\psi) & \mbox{for $x\le1-\psi$};\\
u(t,x+\psi-1)  & \mbox{for $x>1-\psi$},
\end{cases}\quad
\psi\in[0,1).
\]
Note that $T_n^n\mathbf{u}_n(t)=\mathbf{u}_n(t)$.
If $\bar{\mathbf{u}}_n(t)$ is a solution to the KM \eqref{eqn:dsys},
 then so is $T_n^l\bar{\mathbf{u}}_n(t)+\boldsymbol{\theta}$
 for any $\theta\in\Sset^1$ and $l\in[n]$
 by the rotation and translation symmetry  of \eqref{eqn:dsys}.
Indeed, assume that $\bar{\mathbf{u}}_n(t)$ is a solution to the KM \eqref{eqn:dsys}.
Obviously, $\bar{\mathbf{u}}_n(t)+\boldsymbol{\theta}$ is a solution
 to the KM \eqref{eqn:dsys}.
Moreover, applying $T_n^l$ to both side of \eqref{eqn:dsys}
 with $\mathbf{u}_n(t)=\bar{\mathbf{u}}_n(t)$ and $\omega=0$, we obtain
\begin{align*}
\frac{\d}{\d t}\bar{u}_{i+k}(t)
=&\frac{p}{n\alpha_{1n}}\Biggl(\sum_{j=1}^{n-l}
 \sin\left(\bar{u}_{j+l}^n(t)-\bar{u}_{i+k}^n(t)\right)\notag\\
 &\quad
+\sum_{j=n-l+1}^{n}
 \sin\left(\bar{u}_{j-n+l}^n(t)-\bar{u}_{i+k}^n(t)\right)\Biggr)\notag\\
&
-\frac{K}{n \alpha_{2n}}\Biggl(\sum_{j=1}^{n-l}w_{ij}^{2n}
 \sin 2\left(\bar{u}_{j+l}^n(t)-\bar{u}_{i+k}^n(t)\right)\notag\\
&\quad
 +\sum_{j=n-l+1}^{n}w_{ij}^{2n}
 \sin 2\left(\bar{u}_{j-n+l}^n(t)-\bar{u}_{i+k}^n(t)\right)\Biggr),\quad
 i\in[n],
\end{align*}
where $k=l$ or $-n+l$, depending on whether $i\le n-l$ or not.
Here we have used the relations $w_{ij}^{1n}=p$
 and $w_{i+l,j+l}^{2n}=w_{ij}^{2n}$, $i,j\in[n]$.
The above expression implies
 that $T_n^l\bar{\mathbf{u}}_n(t)$ is also a solution to the KM \eqref{eqn:dsys}.
Similarly, if $\bar{\mathbf{u}}(t)$ is a solution to the CL \eqref{eqn:csys},
 then so is $T^\psi\bar{\mathbf{u}}(t)+\boldsymbol{\theta}$
 for any $\theta\in\Sset^1$ and $\psi\in[0,1)$.
Let
\[
\U_{n}=\{T_n^l\bar{\mathbf{u}}_n(t)+\boldsymbol{\theta}
 \mid\theta\in\Sset^1,l\in[n]\}
\]
and
\[
\U=\{T^\psi\bar{\mathbf{u}}(t)+\boldsymbol{\theta}\mid\theta\in\Sset^1,\psi\in I\}
\]
denote such families of solutions to \eqref{eqn:dsys} and \eqref{eqn:csys},
 respectively.
We say that $\U_n$ (resp. $\U$) is \emph{stable}
 if solutions starting in its (smaller) neighborhood
 remain in its (larger but still small) neighborhood for $t\ge 0$,
 and \emph{asymptotically stable} if $\U_n$ (resp. $\U$) is stable
 and the distance between such solutions and $\U_n$ (resp. $\U$) converges to zero
 as $t\to\infty$.
We obtain the following result,
 slightly modifying the proofs of Theorem~2.7 in \cite{IY23} and Theorem~2.3 of \cite{Y24}.

\begin{thm}
\label{thm:2c}
Let $G_{1n}$, $n\in\Nset$, be deterministic and complete.
Suppose that the KM\eqref{eqn:dsys} and CL \eqref{eqn:csys}
 have solutions $\bar{\mathbf{u}}_n(t)$ and $\bar{\mathbf{u}}(t)$, respectively, such that
\begin{equation}
\lim_{n\to\infty}\|\bar{\mathbf{u}}_n(t)-\bar{\mathbf{u}}(t)\|=0
\label{eqn:2c}
\end{equation}
for any $t\in[0,\infty)$.
Then the following hold$:$
\begin{enumerate}
\setlength{\leftskip}{-1.8em}
\item[\rm(i)]
If $\U_n$ is stable $($resp. asymptotically stable$)$ for $n>0$ sufficiently large,
 then $\U$ is also stable $($resp. asymptotically stable$);$
\item[\rm(ii)]
If $\U$ is stable, then for any $\epsilon,\tau>0$ there exists $\delta>0$
 such that for $n>0$ sufficiently large,
 if $\mathbf{u}_n(t)$ is a solution to the KM \eqref{eqn:dsys} satisfying
\begin{equation}
d_n(\mathbf{u}_n(0),\U_n)<\delta,
\label{eqn:2c1}
\end{equation}
then
\begin{equation}
\max_{t\in[0,\tau]}d_n(\mathbf{u}_n(t),\U_n)<\epsilon,
\label{eqn:2c2}
\end{equation}
where
\[
d_n(\mathbf{u}_n(t),\U_n)=\min_{\theta\in\Sset^1}\min_{l\in[n]}
\|\mathbf{u}_n(t)-T_n^l\bar{\mathbf{u}}_n(t)-\boldsymbol\theta\|.
\]
Moreover, if $\U$ is asymptotically stable, then
\begin{equation}
\lim_{t\to\infty}\lim_{n\to\infty}
d_n(\mathbf{u}_n(t),\U_n)=0,
\label{eqn:thm2c}
\end{equation}
where $\mathbf{u}_n(t)$ is any solution to \eqref{eqn:dsys}
 such that $\mathbf{u}_n(0)$ is contained in the basin of attraction for $\U$.
\end{enumerate}
\end{thm}

\begin{rmk}
\label{rmk:2a}
$\U_n$ may not be stable or asymptotically stable in the KM \eqref{eqn:dsys}
 for $n>0$ sufficiently large even if so is $\U$ in the CL \eqref{eqn:csys}.
In the definition of stability and asymptotic stability of solutions to the CL \eqref{eqn:csys},
 we cannot distinguish two solutions that are different only in a set
 with the Lebesgue measure zero.
\end{rmk}

We have the following as a corollary of Theorem~\ref{thm:2c},
 with neither assuming the existence
 of the solution $\bar{\mathbf{u}}_n(t)$ to the KM \eqref{eqn:dsys} satisfying \eqref{eqn:2c}
 nor restricting $G_{1n}$, $n\in\Nset$, to deterministic dense graphs
 (see the proof of Theorem~2.4(ii) and Corollary~2.6 of \cite{Y24}).

\begin{cor}
\label{cor:2a}
Suppose that the CL \eqref{eqn:csys} has a solution $\bar{\mathbf{u}}(t)$
 and $\U=\{T^\psi\bar{\mathbf{u}}(t)+\boldsymbol{\theta}\mid\theta\in\Sset^1,\psi\in I\}$ is stable.
Then for any $\epsilon,\tau>0$ there exists $\delta>0$ such that for $n>0$ sufficiently large,
 if $\mathbf{u}_n(t)$ is a solution to the KM \eqref{eqn:dsys} satisfying
\[
d(\mathbf{u}_n(0),\U)<\delta\quad\mbox{a.s.},
\]
then
\[
\max_{t\in[0,\tau]}d(\mathbf{u}_n(t),\U)<\epsilon\quad\mbox{a.s.}
\]
Moreover, if $\U$ is asymptotically stable, then
\[
\lim_{t\to\infty}\lim_{n\to\infty}d(\mathbf{u}_n(t),\U)=0\quad\mbox{a.s.},
\]
where $\mathbf{u}_n(t)$ is any solution to \eqref{eqn:dsys}
 such that $\mathbf{u}_n(0)$ is contained in the basin of attraction for $\U$ a.s., and
\[
d(\mathbf{u}(t),\U)=\min_{\theta\in\Sset^1,\psi\in[0,1)}
 \|\mathbf{u}(t)-T^\psi\bar{\mathbf{u}}(t)-\mathbf{\theta}\|.
\]
\end{cor}

\begin{rmk}\
\label{rmk:2b}
\begin{enumerate}
\setlength{\leftskip}{-1.8em}
\item[\rm(i)]
In Corollary~$2.6$ of {\rm\cite{Y24}} only complete simple graphs were treated
 but Corollary~{\rm\ref{cor:2a}} can be proven similarly
 since its proof relies only on Theorem~$2.2$ of {\rm\cite{Y24}},
 of which extension to \eqref{eqn:dsys} and \eqref{eqn:csys} is Theorem~{\rm\ref{thm:2b}}.
\item[\rm(ii)]
Corollary~$\ref{cor:2a}$ implies that $\U$ behaves
 as if it is an $($asymptotically$)$ stable family of solutions in the KM \eqref{eqn:dsys}.
\end{enumerate}
\end{rmk}

Finally, we obtain the following results,
 slightly modifying the proofs of Theorems~2.7 and 2.9 in \cite{Y24}.
 
\begin{thm}
\label{thm:2d}
Let $G_{1n}$, $n\in\Nset$, be deterministic and complete.
Suppose that the hypothesis of Theorem~$\ref{thm:2c}$ holds.
Then the following hold$:$
\begin{enumerate}
\setlength{\leftskip}{-1.8em}
\item[\rm(i)]
If $\U_n$ is unstable for $n>0$ sufficiently large
 and no stable family of solutions to the KM \eqref{eqn:dsys}
 converges to $\U$  as $n\to\infty$,
 then $\U$ is unstable$;$
\item[\rm(ii)]
If $\U$ is unstable, then so is $\U_n$ for $n>0$ sufficiently large.
\end{enumerate}
\end{thm}

\begin{thm}
\label{thm:2e}
If $\U$ is unstable,
 then for any $\epsilon,\delta>0$
 there exists $\tau>0$ such that for $n>0$ sufficiently large
\[
d(\mathbf{u}_n(t),\U)>\epsilon\quad\mbox{a.s.},
\]
where $\mathbf{u}_n(t)$ is a solution to the KM \eqref{eqn:dsys} satisfying
\[
d(\mathbf{u}_n(0),\U)<\delta\quad\mbox{a.s.}
\]
\end{thm}

\begin{rmk}\
\label{rmk:2c}
\begin{enumerate}
\setlength{\leftskip}{-1.6em}
\item[\rm(i)]
Only under the hypothesis of Theorem~{\rm\ref{thm:2c}},
 $\U$  is not necessarily unstable
  even if $\U_n$ is unstable for $n>0$ sufficiently large.
Moreover, $\U$ may be asymptotically stable
 even if $\U_n$ is unstable for $n>0$ sufficiently large.
Such an example was found in {\rm\cite{Y24}}.
\item[\rm(ii)]
In Theorem~$2.9$ of {\rm\cite{Y24}} only complete simple graphs were treated
 but Theorem~{\rm\ref{thm:2e}} can be proven similarly,
 like Corollary~{\rm\ref{cor:2a}}, as stated in Remark~{\rm\ref{rmk:2b}(i)}.
\item[\rm(iii)]
Theorem~$\ref{thm:2e}$ implies that $\U$ behaves
 as if it is an unstable family of solutions in the KM \eqref{eqn:dsys}.
\end{enumerate}
\end{rmk}

Thus, the relationship between the KM \eqref{eqn:dsys} and CL \eqref{eqn:csys} is subtle.
However, under the hypothesis of Corollary~\ref{cor:2a},
 if $\U$ is asymptotically stable,
 then a solution to the KM \eqref{eqn:dsys} starting in the basin of attraction of $\U$
 stays near $\U$ for $n,t>0$ sufficiently large.
This indicates that the ``asymptotic stability'' of $\U$ is observed
 in numerical simulations although they can be performed
 only for large finite values of $n,t>0$.
We will observe this behavior in Section~5.


\section{Linear Eigenvalue Problem}

We turn to direct discussion on the KM \eqref{eqn:dsys} and CL \eqref{eqn:csys}.
We begin with the eigenvalue problem
 for the linearized equation of the CL \eqref{eqn:csys} with $\omega=0$.
We first note that by its rotation and translation symmetry,
 if $u(t,x)$ is a solution to the CL \eqref{eqn:csys},
 then so is $u(t,x+\psi)+\theta$ for any $\psi\in I$ and $\theta\in\Sset^1$
 (see the statement just before Theorem~\ref{thm:2c}).

\begin{figure}
\includegraphics[scale=0.56]{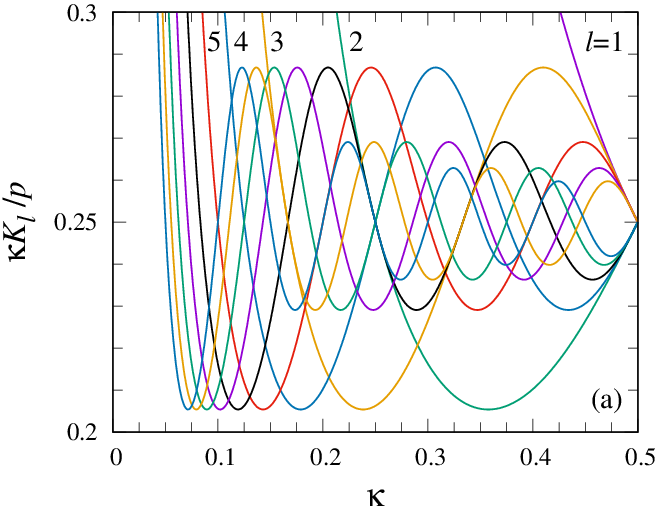}\
\includegraphics[scale=0.56]{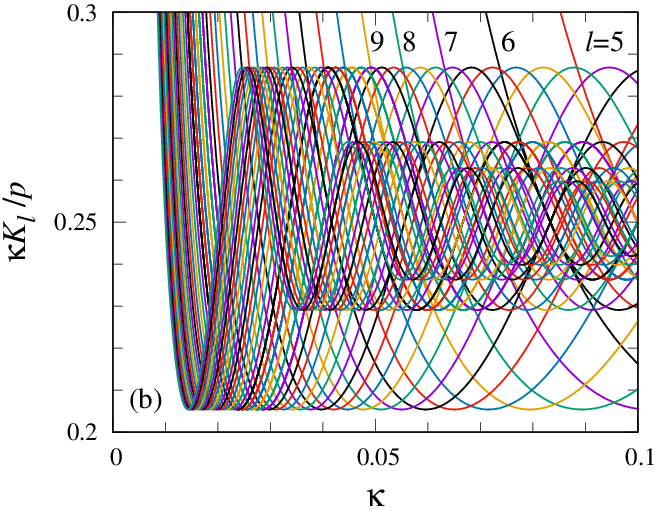}
\caption{Dependence of $K_l$ given by \eqref{eqn:Kl} on $\kappa$:
(a) $l=1$-$10$; (b) $l=5$-$50$.
\label{fig:3a}}
\end{figure}

Following Section~5 of \cite{IY23},
 we discuss the eigenvalue problem
 for the linear operator $\L:L^2(I)\to L^2(I)$ given by
\begin{align}
\L \phi(x)
=& \int_I(p-2K W_2(x,y))(\phi(y)-\phi(x))\d y\notag\\
=& \int_I (p-2K W_2(x,y)) \phi(y)\d y + (4K\kappa-p)\phi(x)
\label{eqn:lep}
\end{align}
for the linearization of \eqref{eqn:csys} around $u(t,x)=0$.
Note that
\begin{equation*}
\int_I W_2(x,y)\d y=2\kappa.
\end{equation*}
We see that $\phi(x)=1$ is an eigenfunction for the zero eigenvalue and
\begin{equation*}
\phi(x)=\cos 2 \pi\ell x,\quad
\sin 2\pi\ell x
\end{equation*}
are eigenfunctions for the eigenvalue
\[
\lambda=4K\kappa-p-\frac{2K}{\pi\ell} \sin 2 \pi\ell\kappa
\]
for each $\ell\in\Nset$ since
\begin{align*}
&
\int_I W_2(x,y) \sin 2 \pi\ell y\,\d y
=\frac{1}{\pi\ell}\sin 2 \pi\ell \kappa\,\sin 2 \pi\ell x,\\
&
\int_I W_2(x,y) \cos 2 \pi\ell y\,\d y
=\frac{1}{\pi\ell}\sin 2 \pi\ell \kappa\,\cos 2 \pi\ell x.
\end{align*}
These eigenvalues are the only ones of $\L$
 since the Fourier expansion of any function in $L^2(I)$ converges to itself a.e.
 by Carleson's theorem \cite{C66}.
Thus, if
\begin{equation}
K=K_\ell:=\frac{\pi\ell p}{2(2\pi\ell\kappa-\sin 2\pi\ell\kappa)}
\label{eqn:Kl}
\end{equation}
for some $\ell\in\Nset$, then the zero eigenvalue of $\L$ is of geometric multiplicity three
 and a bifurcation may occur in the CL \eqref{eqn:csys}.
See Fig.~\ref{fig:3a} for the dependence of $K_\ell$ given by \eqref{eqn:Kl} on $\kappa$
 for several values of $\ell$.
 
Let $\zeta_0$ be the smallest $\zeta>0$ such that the function
\[
\varphi(\zeta)=\frac{\sin\zeta}{\zeta}
\]
takes a local minimum, and let $\varphi_0=-\varphi(\zeta_0)$.
We estimate $\zeta_0=4.4934\ldots$ and $\varphi_0=0.21723\ldots$
We easily prove the following for $K_\ell$.

\begin{prop}\
\label{prop:3a}
\begin{enumerate}
\setlength{\leftskip}{-1.6em}
\item[\rm(i)]
If $\kappa\in(1/2\ell,1/\ell)$, then
\[
K_\ell<\frac{p}{4\kappa}
\]
for $\ell>1$.
In particular, $\min_{\ell\in\Nset}K_\ell<p/4\kappa$ for any $\kappa\in(0,\tfrac{1}{2})$.
\item[\rm(ii)]
$\displaystyle K_\ell\ge\frac{p}{4\kappa(1+\varphi_0)}$ for $\ell>1$,
 where the equality holds at $\kappa=\zeta_0/2\pi\ell$.   
\item[\rm(iii)]
$K_\ell=\tfrac{1}{2}p$ for $\ell\in\Nset$ when $\kappa=\tfrac{1}{2}$.
\end{enumerate}
\end{prop}

\begin{proof}
We rewrite \eqref{eqn:Kl} as
\begin{equation}
\frac{p}{4K_\ell\kappa}=1-\varphi(2\pi\ell\kappa).
\label{eqn:prop3a}
\end{equation}
We easily see that if $\kappa\in(1/2\ell,1/\ell)$,
 then $2\pi\ell\kappa\in(\pi,2\pi)$, so that the right-hand side of \eqref{eqn:prop3a}
 is larger than one.
Noting that
\[
(0,\tfrac{1}{2})=\bigcup_{\ell=2}^\infty(1/2\ell,1/\ell),
\]
we obtain part~(i).
On the other hand,
 we have $\varphi(\zeta)\ge\varphi(\zeta_0)=-\varphi_0$ on $(0,3\pi]$ and\[
\varphi(\zeta)>-\frac{1}{(2j+1)\pi}>-\varphi_0
\]
on $((2j+1)\pi,(2j+3)\pi]$, $j\in\Nset$.
This yields part~(ii) by \eqref{eqn:prop3a}.
Part~(iii) also immediately follows from \eqref{eqn:prop3a}
 since $\varphi(\pi\ell)=0$.
\end{proof}

We notice that $4K_1\kappa>p$ for $\kappa\in(0,\tfrac{1}{2})$ since
\[
\frac{p}{4K_1\kappa}=1-\frac{\sin 2\pi\kappa}{2\pi\kappa}<1.
\]
Thus, $K_1$ cannot be the minimum of $K_j$, $j\in\Nset$, as seen in Fig.~\ref{fig:3a}.

On the other hand, we obtain
\begin{align}
\langle \L \phi, \phi \rangle
=&
\int_{I^2}(p-2K W_2(x,y))(\phi(y)-\phi(x))\phi(x)\d x\d y\notag\\
=&\tfrac{1}{2} \int_{I^2} (2KW_2(x,y)-p) (\phi(x)-\phi(y))^2\d x\d y,
\label{eqn:prop3b}
\end{align}
where $\langle\cdot,\cdot\rangle$ represents the inner product in $L^2(I)$.
Here we have used the relation
\[
\int_I(p-2K W_2(x,y))\phi(x)^2\d x\d y
 =\int_I(p-2K W_2(x,y))\phi(y)^2\d x\d y,
\]
which follows from the symmetry of $W_2(x,y)$.
Noting that if $K\leq 0$ and $\phi(x)$ is not a constant function,
 then Eq.~\eqref{eqn:prop3b} is negative, we obtain the following.

\begin{prop}
\label{prop:3b}
If $K\leq 0$, then the stationary solution $u(t,x)=0$ is linearly stable
 and the one-parameter family of stationary solutions
 $\U_0=\{u=\theta\mid\theta\in\Sset^1\}$ is asymptotically stable
  in the CL \eqref{eqn:csys}.
\end{prop}


\section{Bifurcations}

We now take the coupling constant $K$ for the second mode interaction as a control parameter,
 and analyze bifurcations in the CL \eqref{eqn:csys} with $\omega=0$.
It follows from the analysis of Section~3 that
 they occur at $K=K_\ell$, $\ell\in\Nset$.
However, the zero eigenvalue of the related linear operator $\L$
 is of geometric multiplicity three and very degenerate at $K=K_\ell$,
 due to the rotation and translation symmetry,
 as stated at the beginning of Section~3,
 so that some special treatments are required. 
So we first +only consider solutions $u(t,x)$ in the CL \eqref{eqn:csys} such that
\begin{equation}
\int_I u(t,x)\d x=0
\label{eqn:con}
\end{equation}
since $u(t,x)+\theta$ is also a solution for any $\theta\in\Sset^1$.
This reduces the degeneracy of multiplicity three  by one.

Consider the case of $K\approx K_\ell$ for some $\ell\in\Nset$
 and introduce a parameter
\begin{equation}
\mu=2(2\kappa-\delta_\ell)(K-K_\ell)\approx 0,
\label{eqn:mu}
\end{equation}
where
\[
\delta_j=\frac{1}{\pi j}\sin 2\pi j\kappa,\quad
j\in\Nset.
\]
Let
\begin{equation}
u(t,x)=\sum_{j=1}^\infty(\xi_j(t)\cos 2\pi jx+\eta_j(t)\sin 2\pi jx),
\label{eqn:solex}
\end{equation}
which satisfies \eqref{eqn:con}.
We substitute \eqref{eqn:solex} into \eqref{eqn:csys},
 and integrate the resulting equation from $x=0$ to $1$
 after multiplying it with $\cos 2\pi jx$ or $\sin 2\pi jx$, $j\in\Nset$, 
 to obtain
\begin{equation}
\begin{split}
&
\dot{\xi}_\ell=\mu\xi_\ell-\beta(\xi_\ell^2+\eta_\ell^2)\xi_\ell+\cdots,\quad
\dot{\eta}_\ell=\mu\eta_\ell-\beta(\xi_\ell^2+\eta_\ell^2)\eta_\ell+\cdots,\\
&
\dot{\xi}_j=-\beta_j\xi_j+\cdots,\quad
\dot{\eta}_j=-\beta_j\eta_j+\cdots,\quad
j\neq\ell,\\
&
\dot{\mu}=0,
\end{split}
\label{eqn:ifex}
\end{equation}
where `$\cdots$' represents higher-order terms of
\[
O\left(|\xi_\ell|^5+|\eta_\ell|^5+\sum_{j=1,j\neq \ell}^\infty(|\xi_j|^3+|\eta_j|^3)+\mu^2\right)
\]
for the first two equations and
\[
O\left(\sum_{j=1}^\infty(|\xi_j|^3+|\eta_j|^3)+\mu^2\right)
\]
for the third and fourth equations,
\begin{equation}
\beta=\tfrac{13}{8}p-K_\ell(\kappa-\delta_{2\ell})
\label{eqn:beta}
\end{equation}
and
\begin{align*}
&
\beta_j=p-2K_\ell(2\kappa-\delta_j)
=2K_\ell(\delta_j-\delta_\ell),\quad
j\in\Nset\setminus\{\ell\}.
\end{align*}
Here we have used the relation $p=2K_\ell(2\kappa-\delta_\ell)$.
See Appendix~A for the derivation of \eqref{eqn:ifex}.

\begin{figure}
\includegraphics[scale=0.6]{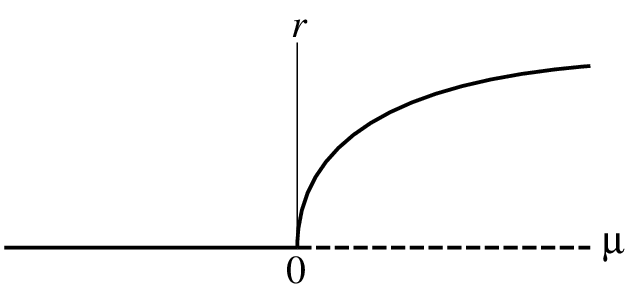}
\caption{Bifurcation diagram for \eqref{eqn:r}.
The solid and broken lines represent stable and unstable equilibria, respectively.
\label{fig:4a}}
\end{figure}

Suppose that $\beta_j\neq 0$ for any $j\neq\ell$.
Then the origin in the infinite-dimensional system \eqref{eqn:ifex} is an equilibrium
 which has a three-dimensional center manifold given by
\begin{equation}
W^\c
=\{(\xi_i,\eta_i), i\in\Nset\mid
 (\xi_j,\eta_j)=O(|\xi_\ell|^3+|\eta_\ell|^3+\mu^2),j\neq\ell\}.
\label{eqn:Wc}
\end{equation}
By the center manifold reduction \cite{HI11},
 the system \eqref{eqn:ifex} is reduced to
\begin{equation}
\begin{split}
&
\dot{\xi}_\ell=\mu\xi_\ell-\beta(\xi_\ell^2+\eta_\ell^2)\xi_\ell+\cdots,\quad
\dot{\eta}_\ell=\mu\eta_\ell-\beta(\xi_\ell^2+\eta_\ell^2)\eta_\ell+\cdots,\\
&
\dot{\mu}=0
\end{split}
\label{eqn:cmex}
\end{equation}
on the center manifold, where `$\cdots$' represents higher-order terms of
\[
O\left(|\xi_\ell|^5+|\eta_\ell|^5+\mu^2\right).
\]
See Appendix~B for the validity of application of the theory in our setting.
We easily see that
 the origin $(\xi_\ell,\eta_\ell)=(0,0)$ is always an equilibrium in \eqref{eqn:cmex}.
Letting $r=\sqrt{\xi_\ell^2+\eta_\ell^2}\ge 0$,
 we rewrite \eqref{eqn:cmex}.as
\begin{equation}
\dot{r}=\mu r-\beta r^3+O(r^5+\mu^2),\quad
\dot{\mu}=0.
\label{eqn:r}
\end{equation}
Here by the translation symmetry,
 the first equation of \eqref{eqn:r} must depend only on $r$ and $\mu$
 even if the higher-order terms are included,
 since the symmetry would be broken otherwise.
We easily show the following for \eqref{eqn:r} without the higher-order terms
 when $\beta>0$:
\begin{enumerate}
\setlength{\leftskip}{-1.8em}
\item[(i)]
If $\mu<0$, then the equilibrium $r=0$ is stable;
\item[(ii)]
If $\mu>0$, then the equilibrium $r=0$ is unstable
 and there exists another stable equilibrium at
\[
r=\sqrt{\frac{\mu}{\beta}}.
\]
\end{enumerate}
See Fig.~\ref{fig:4a} for the bifurcation diagram for \eqref{eqn:r}.
From this result, we obtain the following for the CL \eqref{eqn:csys}
 when the restriction \eqref{eqn:con} is removed.

\begin{thm}
\label{thm:4a}
Fix the value of $\kappa$
 and suppose that $K_\ell$ is the unique minimum of $K_j$, $j\in\Nset$.
Then a bifurcation of the one-parameter family of stationary solutions
 \[
 \U^0=\{u=\theta\mid\theta\in\Sset^1\}
 \]
occurs at $K=K_\ell$ in the CL \eqref{eqn:csys} as follows$:$
\begin{enumerate}
\setlength{\leftskip}{-1.8em}
\item[(i)]
When $K<K_\ell$ near $K=K_\ell$, the family $\U^0$ of stationary solutions
   is stable$;$
\item[(ii)]
When $K>K_\ell$ near $K=K_\ell$, $\U^0$ is unstable
 and  there exists another stable two-parameter family of stationary solutions
\[
\U^\ell=\left\{u=\sqrt{\frac{\mu}{\beta}}\sin(2\pi\ell x+\psi)+\theta+O(\mu)
 \mid\psi,\theta\in\Sset^1\right\}.
\]
\end{enumerate}
Here $\mu=O(K-K_\ell)$ and $\beta=O(1)$
 are given in \eqref{eqn:mu} and \eqref{eqn:beta}, respectively,
 and $\beta>0$.
\end{thm}
 
\begin{proof}
It remains to show that $\beta>0$ and $\beta_j>0$ for any $j\neq\ell$,
 but they immediately follow since by Proposition~\ref{prop:3a}(i)
\[
\beta>\tfrac{1}{8}(11p+2(p-4K\kappa_\ell))>0
\]
and
\[
\beta_j=4\kappa(K_j-K_\ell)\left(1-\frac{\sin 2\pi j\kappa}{2\pi j\kappa}\right)>0
\]
for any $j\neq\ell$.
\end{proof}
 
\begin{rmk}\
\label{rmk:4a}
\begin{enumerate}
\setlength{\leftskip}{-1.6em}
\item[\rm(i)]
It follows from Proposition~{\rm\ref{prop:3a}(ii)}
 that the hypothesis of Theorem~$\ref{thm:4a}$ holds,
 i.e., $K_\ell$ is the unique minimum of $K_j$, $j\in\Nset$,
 if $\kappa$ is close to $\zeta_0/ 2\pi\ell$ for some integer $\ell>1$.
\item[\rm(ii)]
For each $\ell\in\Nset$,
 a bifurcation similar to one detected in Theorem~$\ref{thm:4a}$ also occurs
 at $K=K_\ell$ even if $K_\ell$ is not the unique minimum of $K_j$, $j\in\Nset$,
 but $K_\ell\neq K_j$ for any $j\neq\ell$,
 although the two-parameter family of $\ell$-humped solutions born there is unstable.
\item[\rm(iii)]
From  Corollary~{\rm\ref{cor:2a}} and Theorems~$\ref{thm:2e}$ and $\ref{thm:4a}$
 we see that $\U^0$ and $\U^\ell$ behave
 as if they are asymptotically stable or unstable families of solutions in the KM~\eqref{eqn:dsys}
 near $K=K_\ell$ for $n>0$ sufficiently large.
Thus, the KM \eqref{eqn:dsys} suffers a ``bifurcation''
 similar to one detected in Theorem~$\ref{thm:4a}$ for the CL \eqref{eqn:csys}.
\end{enumerate}
\end{rmk}


\section{Numerical Simulations}

\begin{figure}[t]
\begin{minipage}[t]{0.495\textwidth}
\begin{center}
\includegraphics[scale=0.3]{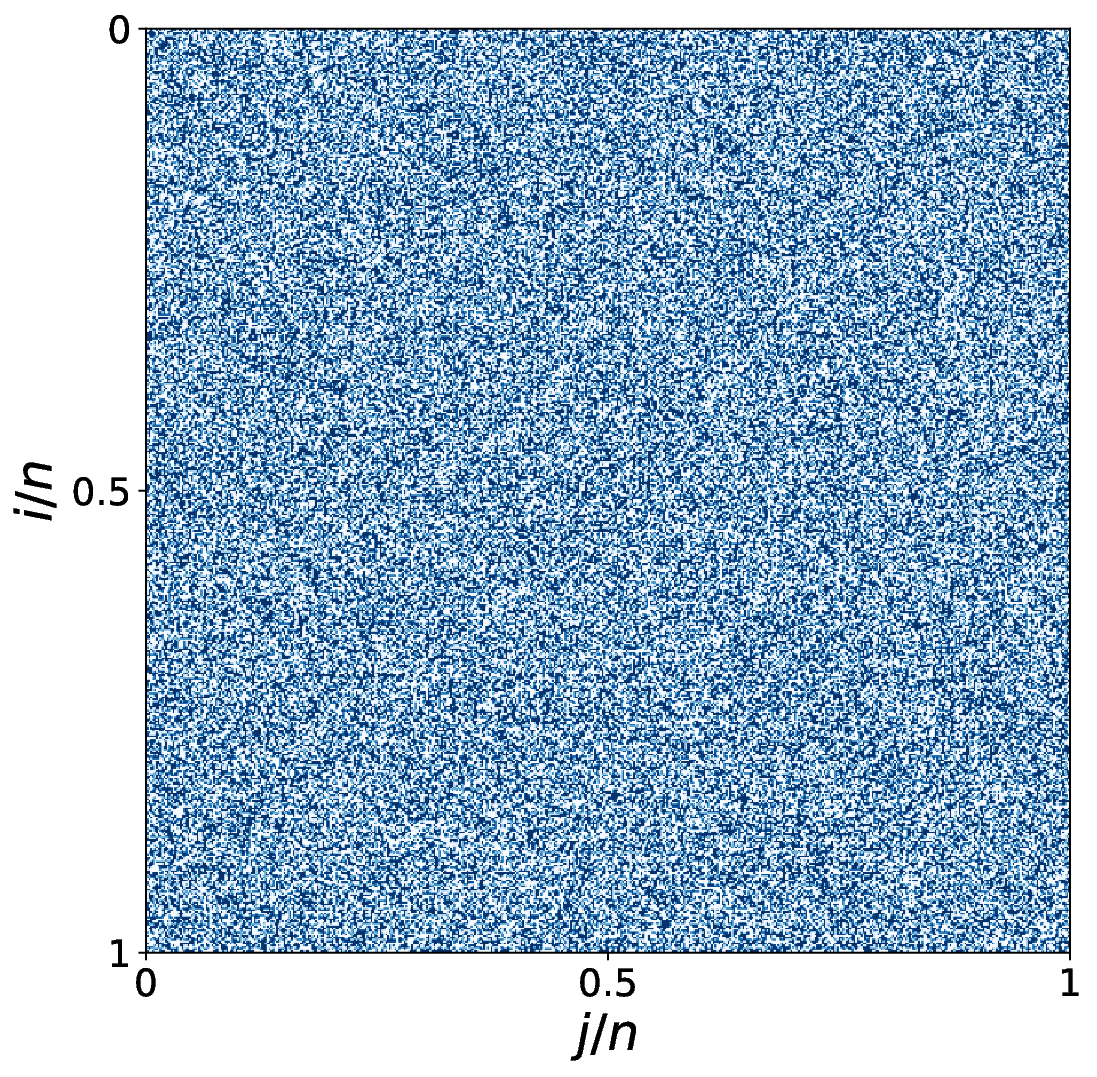}\\
{\footnotesize(a)}
\end{center}
\end{minipage}
\begin{minipage}[t]{0.495\textwidth}
\begin{center}
\includegraphics[scale=0.3]{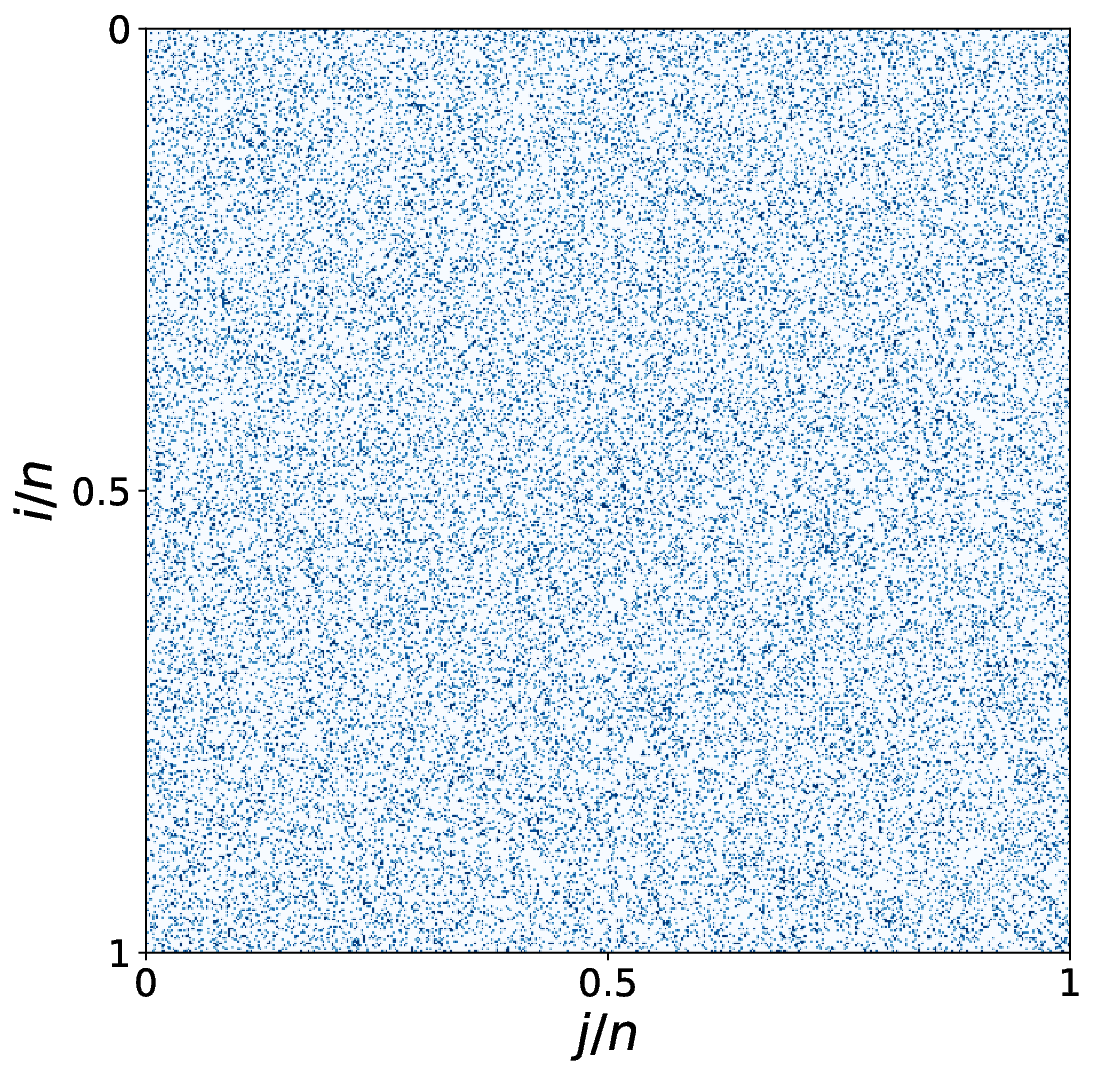}\\
{\footnotesize(b)}
\end{center}
\end{minipage}
\caption{Pixel pictures of sampled weight matrices of the random undirected graphs
 given by  $w_{ij}^n=1$, $i,j\in[n]$, with probability \eqref{eqn:prob1} and  \eqref{eqn:prob2}
 for $n=500$:
(a) Dense graph with $p=0.5$;
(b) sparse graph with $p=1$ and $\gamma=0.3$.
The color of the corresponding pixel is blue if $w_{ij}^n=1$
 and it is light blue otherwise.
\label{fig:5a}}
\end{figure}

Finally, we give numerical simulation results for the KM \eqref{eqn:dsys} with $\omega=0$.
As stated in Section~1 and explained in Section~2,
 the KM \eqref{eqn:dsys} is well approximated by the CL \eqref{eqn:csys},
 when the weight matrices $W(G_{kn})$ and graphons $W_k$, $k=1,2$,
 have the relation \eqref{eqn:ddg}, \eqref{eqn:rdg} or \eqref{eqn:rsg},
 depending on whether the graph $G_{kn}$, $k=1,2$,
 are deterministic dense, random dense or random sparse.

We consider the following three cases for $G_{1n}$:
\begin{enumerate}
\setlength{\leftskip}{-1.8em}
\item[(i)]
Complete simple graph, which is deterministic dense with $p=1$;
\item[(ii)]
Random undirected dense graph in which $w_{ij}^n=1$ with probability
\begin{equation}
\Pset(j\sim i)=p,\quad
i,j\in[n];
\label{eqn:prob1}
\end{equation}
\item[(iii)]
Random undirected space graph in which $w_{ij}^n=1$ with probability
\begin{equation}
\Pset(j\sim i)=n^{-\gamma}p,\quad
i,j\in[n].
\label{eqn:prob2}
\end{equation}
\end{enumerate}
Note that $\alpha_{2n}^{-1}=n^\gamma>p\in(0,1]$ for $\gamma\in(0,\tfrac{1}{2})$.
We specifically take $p=0.5$ in case~(ii) and $p=1$ and $\gamma=0.3$ in case~(iii).
We choose a deterministic $\lfloor n\kappa\rfloor$-nearest neighbor graph as $G_{2n}$:
\begin{equation*}
w_{ij}^2=
\begin{cases}
1 & \mbox{if $|i-j|\leq\kappa n$ or $|i-j|\geq (1-\kappa)n$;}\\
0 & \mbox{otherwise}
\end{cases}
\end{equation*}
for the weight matrix in the KM \eqref{eqn:dsys}, as in Section~5 of \cite{IY23}.
Some numerical simulation results for case~(i) were given in Section~5 of \cite{IY23},
 as stated in Section~1.
Figure~\ref{fig:5a} provides the weight matrices
 for numerically computed samples
 of the random undirected dense and sparse graphs with $n=500$
 and $p=0.5$ or $1$.

\begin{figure}[t]
\begin{minipage}[t]{0.495\textwidth}
\begin{center}
\includegraphics[scale=0.27]{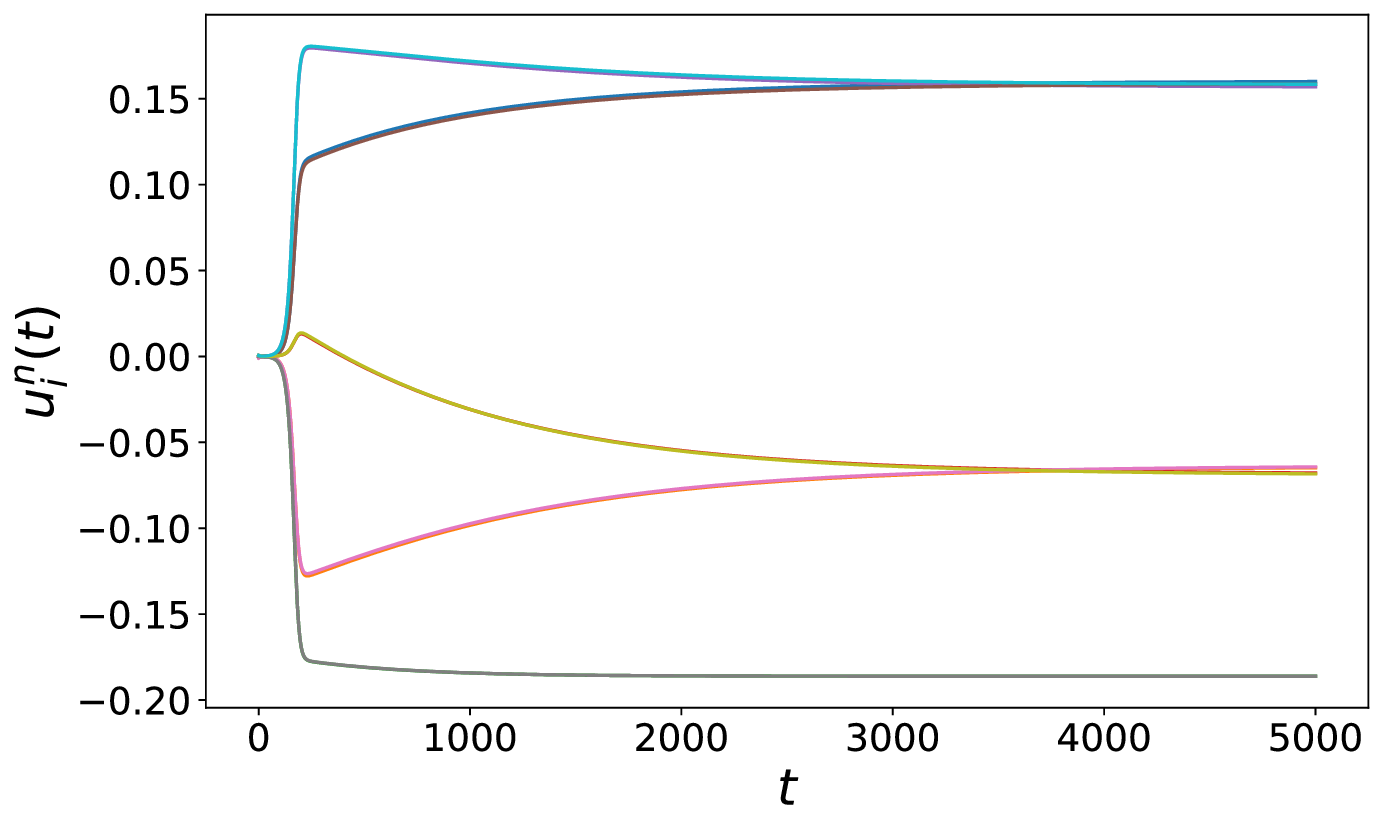}\\[-1ex]
{\footnotesize(a)}
\end{center}
\end{minipage}
\begin{minipage}[t]{0.495\textwidth}
\begin{center}
\includegraphics[scale=0.27]{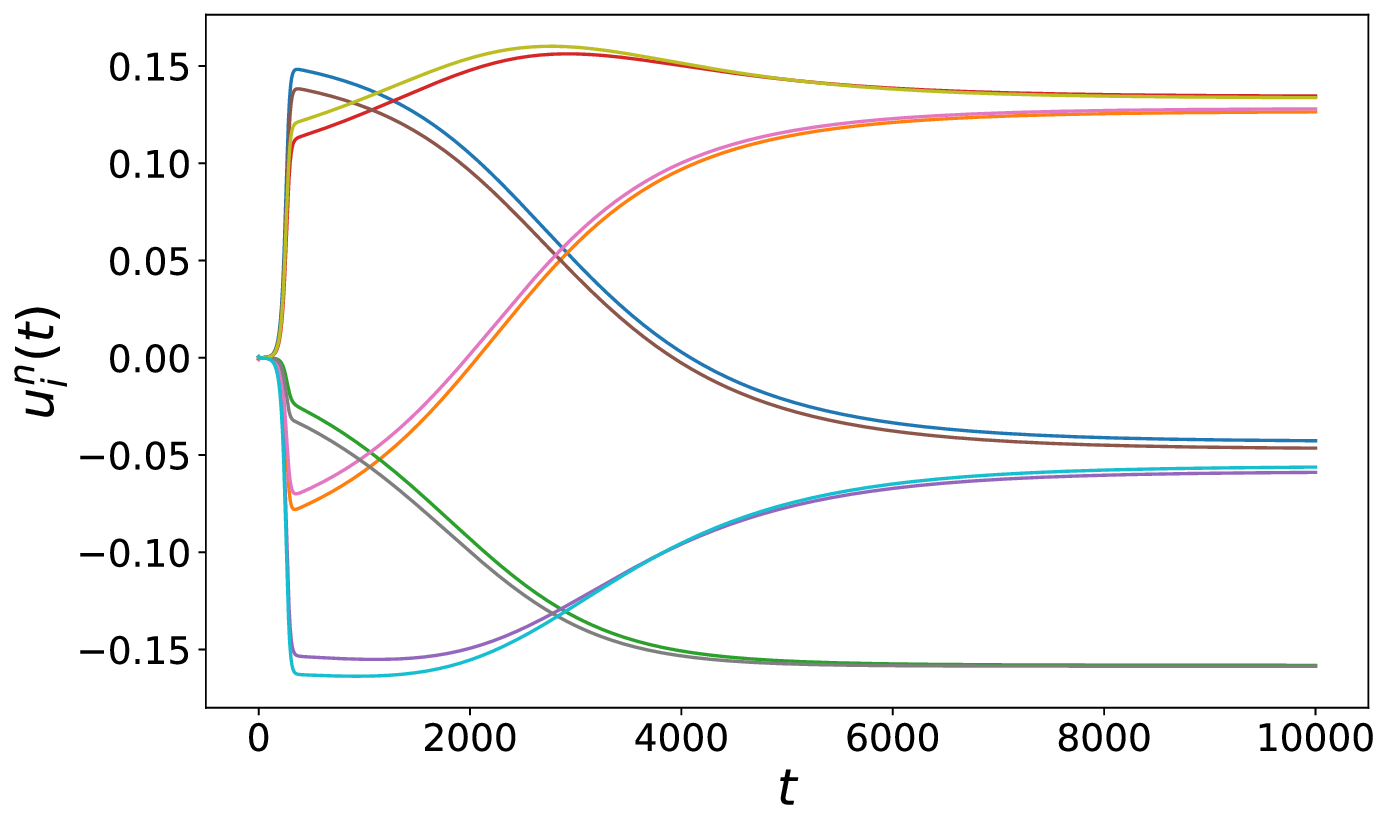}\\[-1ex]
{\footnotesize(b)}
\end{center}
\end{minipage}
\begin{center}
\includegraphics[scale=0.27]{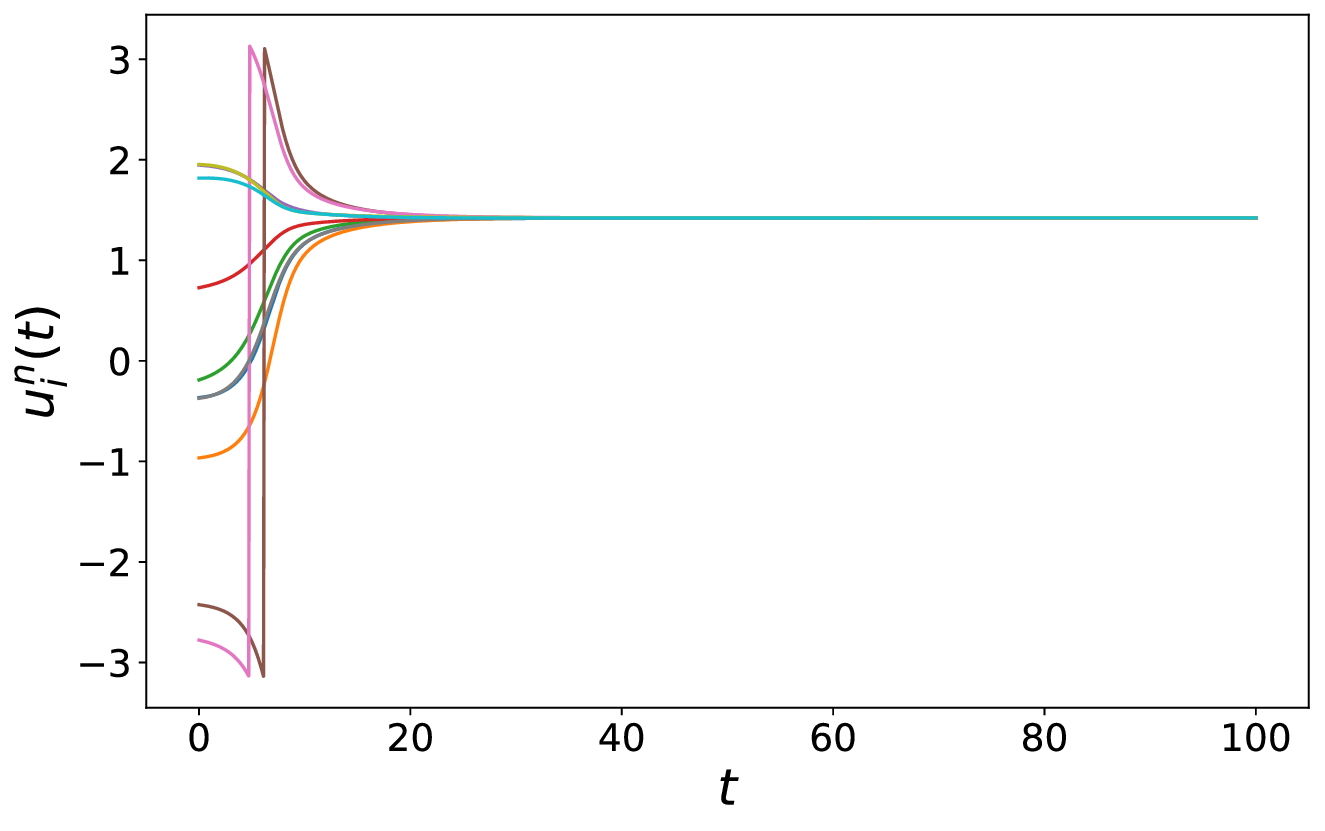}\\[-1ex]
{\footnotesize(c)}
\end{center}
\vspace*{-2ex}
\caption{Numerical simulation results for the KM \eqref{eqn:dsys}
 with $\omega=0$ and $n=500$  in case~(i):
(a) $(\kappa,K)=(1/3,0.65)$;  (b) $(1/8,1.7)$; (c) $(1/3,0.6)$.
The time history of every 50th node
 (from 25th to 475th) is plotted with a different color.
\label{fig:5b1}}
\end{figure}

\begin{figure}[t]
\begin{minipage}[t]{0.495\textwidth}
\begin{center}
\includegraphics[scale=0.27]{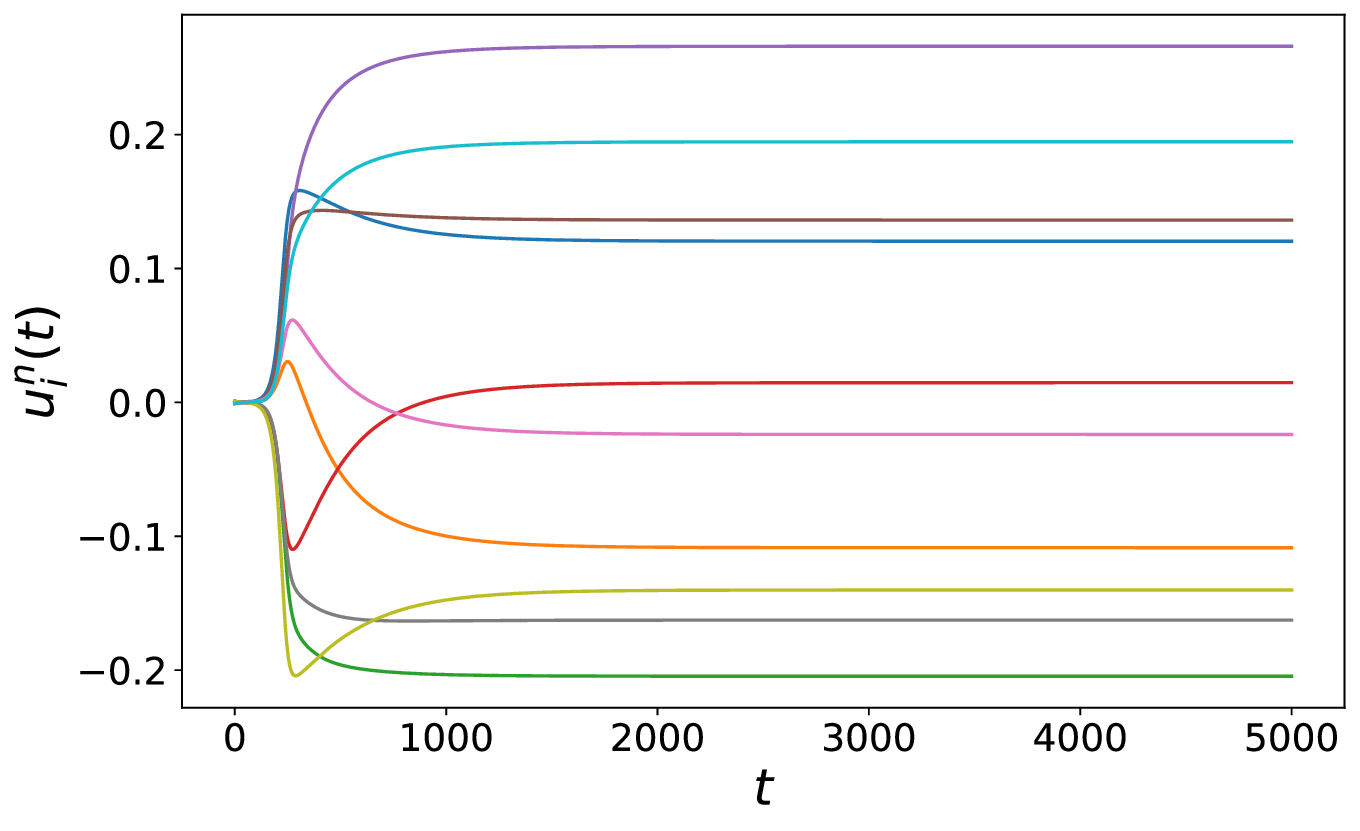}\\[-1ex]
{\footnotesize(a)}
\end{center}
\end{minipage}
\begin{minipage}[t]{0.495\textwidth}
\begin{center}
\includegraphics[scale=0.27]{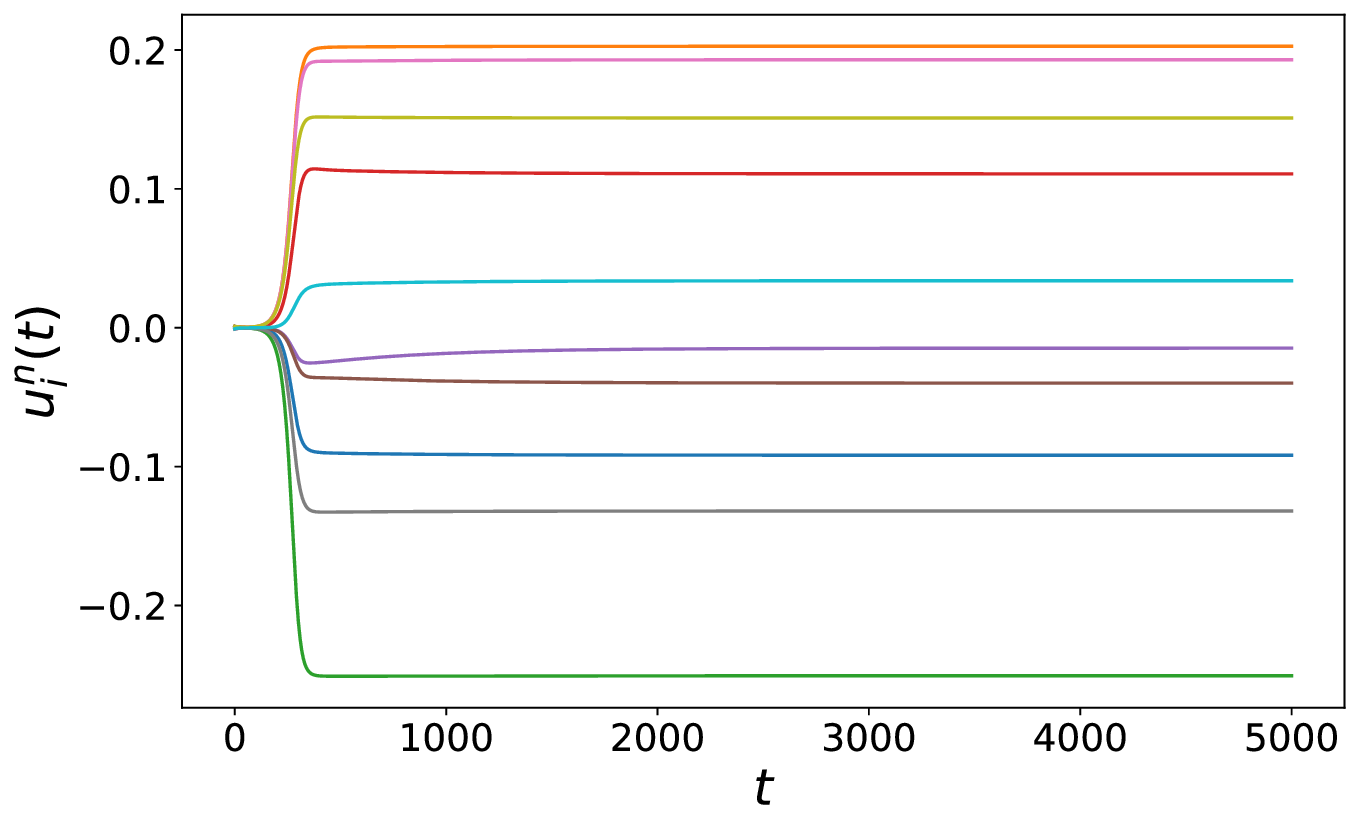}\\[-1ex]
{\footnotesize(b)}
\end{center}
\end{minipage}
\begin{center}
\includegraphics[scale=0.27]{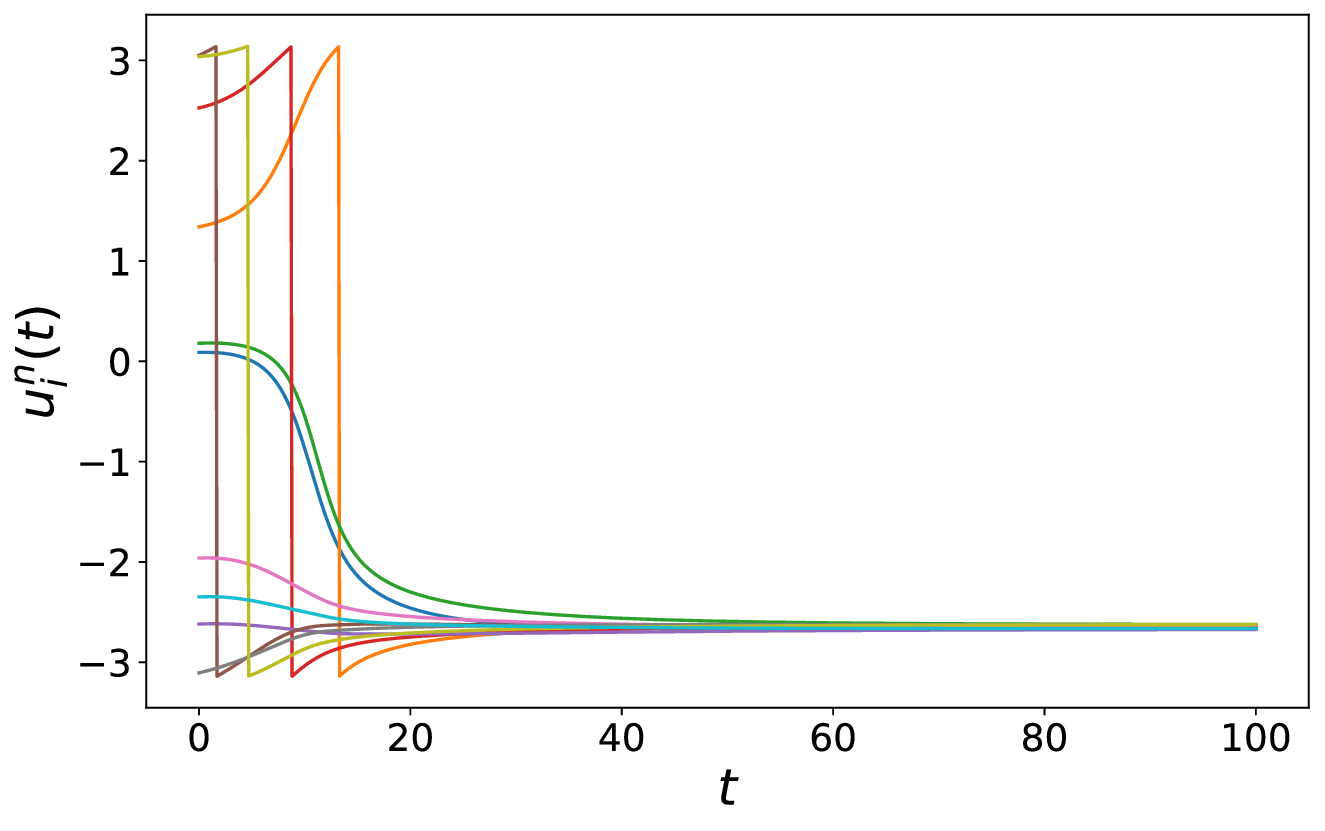}\\[-1ex]
{\footnotesize(c)}
\end{center}
\vspace*{-2ex}
\caption{Numerical simulation results for the KM \eqref{eqn:dsys}
 with $\omega=0$, $n=500$ and $p=0.5$ in case~(ii):
(a) $(\kappa,K)=(1/3,0.325)$;  (b) $(1/8,0.85)$; (c) $(1/3,0.3)$.
See also the caption of Fig.~\ref{fig:5b1}.
\label{fig:5b2}}
\end{figure}

\begin{figure}[t]
\begin{minipage}[t]{0.495\textwidth}
\begin{center}
\includegraphics[scale=0.27]{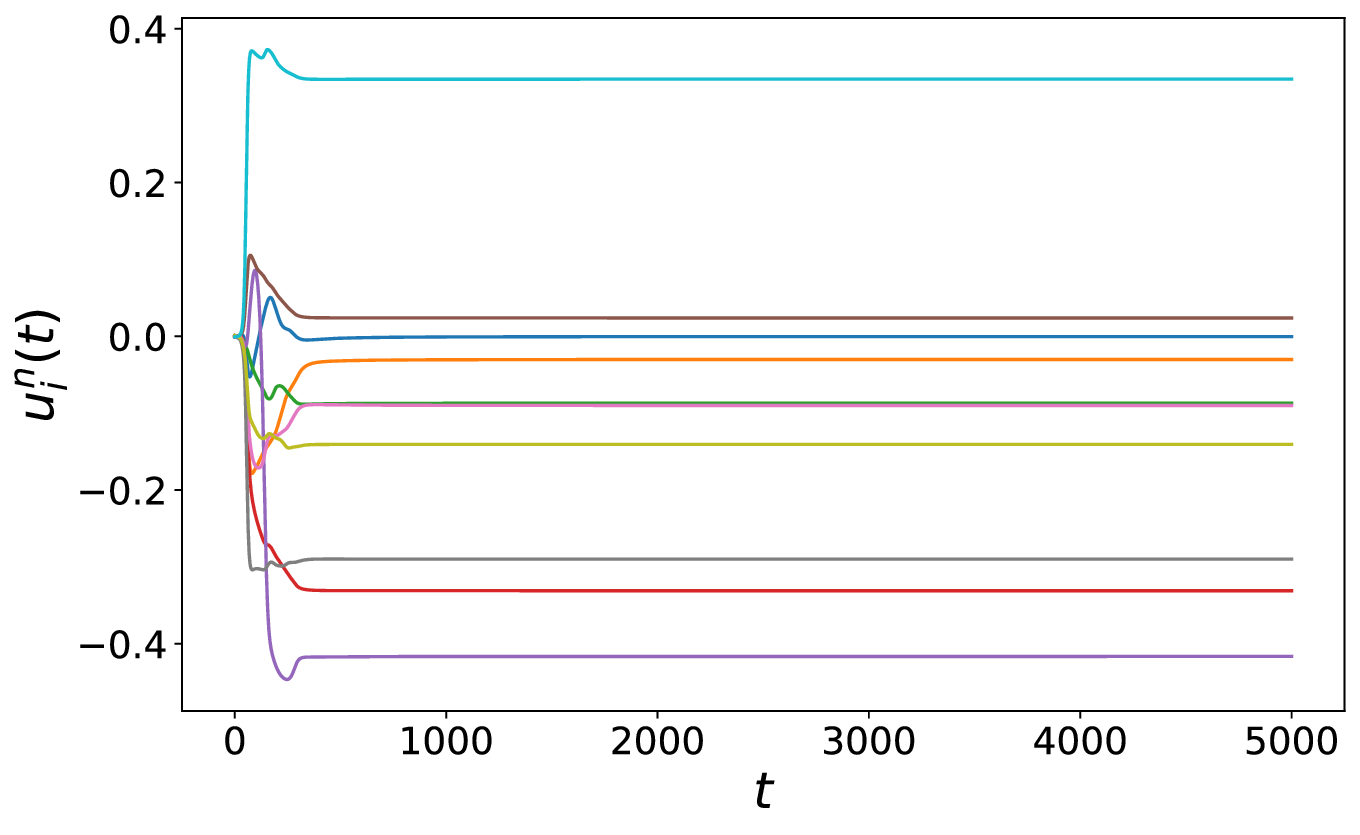}\\[-1ex]
{\footnotesize(a)}
\end{center}
\end{minipage}
\begin{minipage}[t]{0.495\textwidth}
\begin{center}
\includegraphics[scale=0.27]{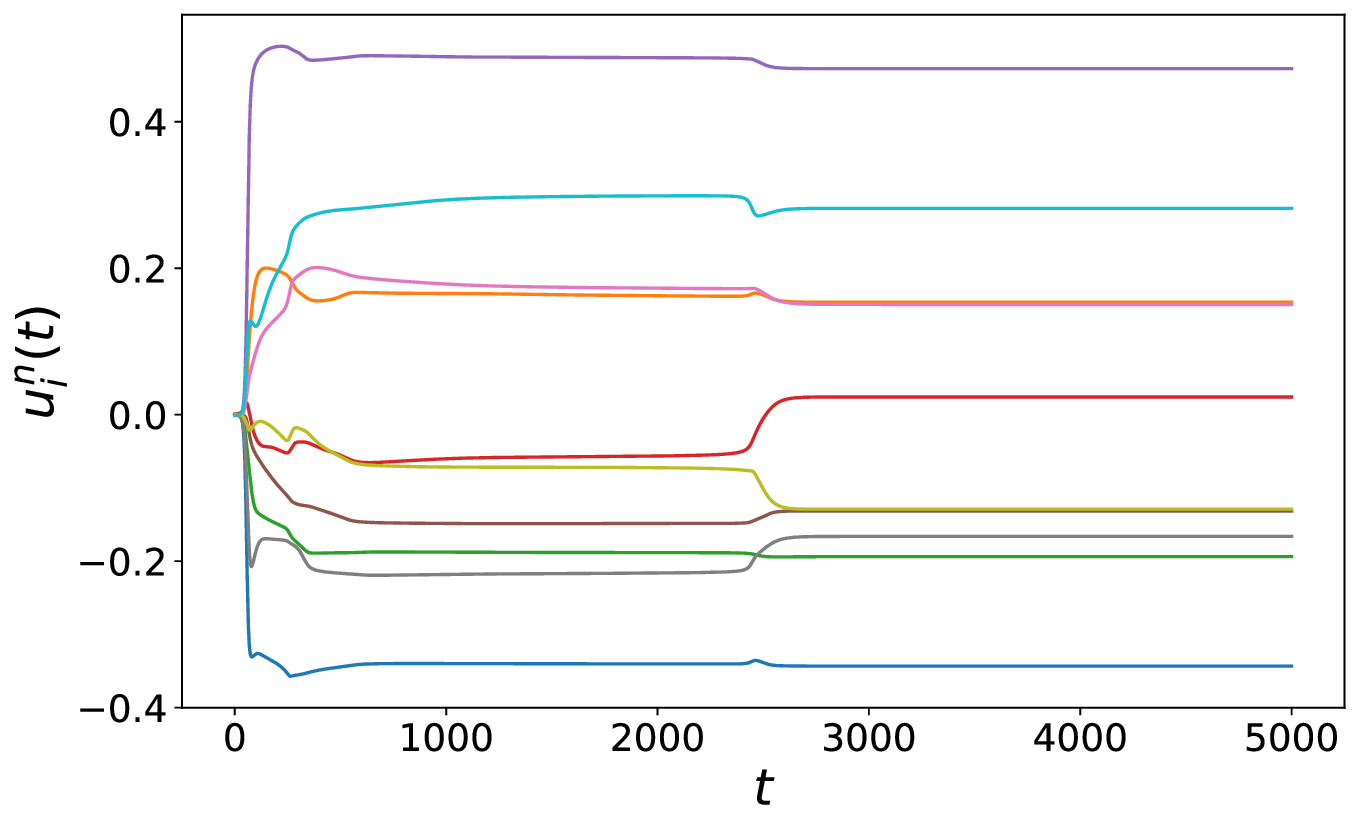}\\[-1ex]
{\footnotesize(b)}
\end{center}
\end{minipage}
\begin{center}
\includegraphics[scale=0.27]{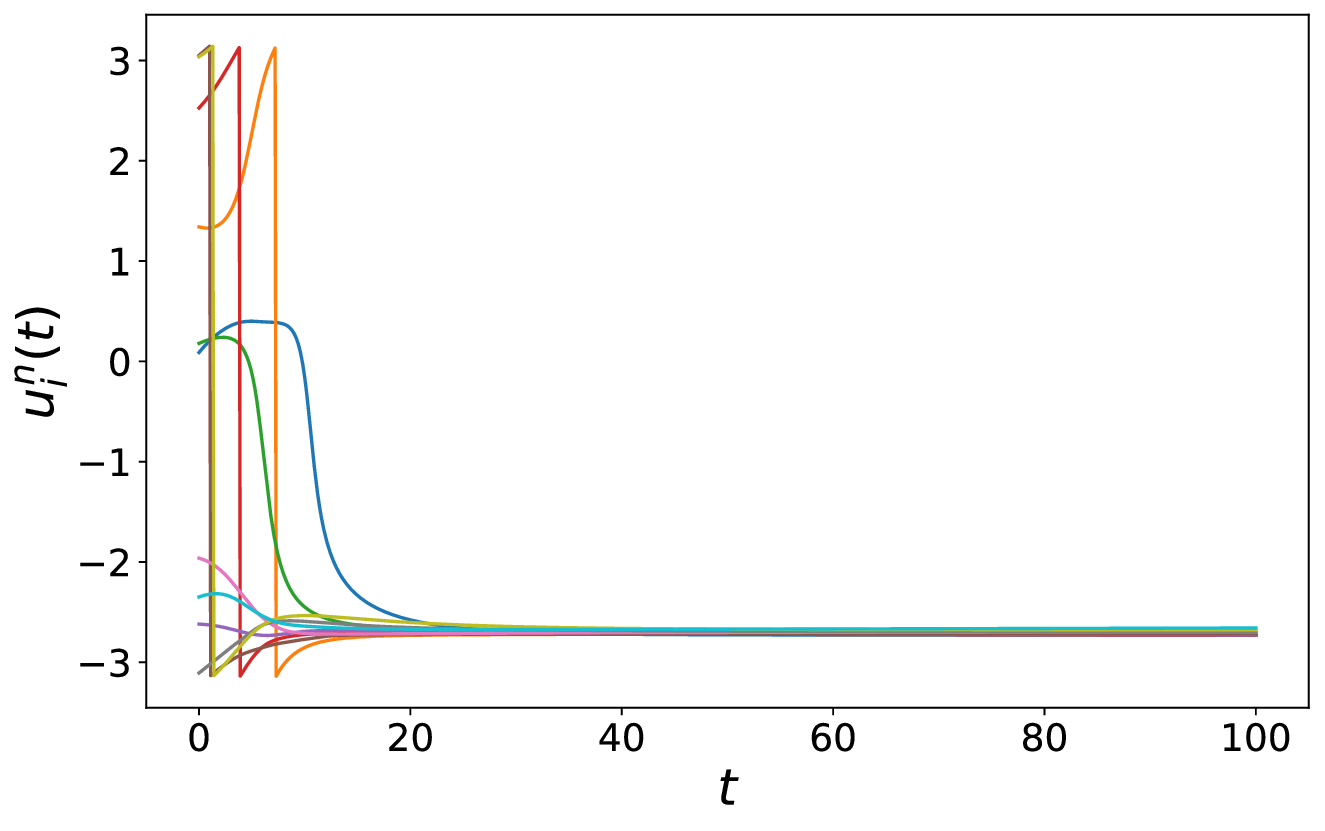}\\[-1ex]
{\footnotesize(c)}
\end{center}
\vspace*{-2ex}
\caption{Numerical simulation results for the KM \eqref{eqn:dsys}
 with $\omega=0$, $n=500$, $p=1$ and $\gamma=0.3$ in case~(iii):
(a) $(\kappa,K)=(1/3,0.65)$; 
(b) $(1/8,1.7)$;
(c) $(1/3,0.54)$.
See also the caption of Fig.~\ref{fig:5b1}.
\label{fig:5b3}}
\end{figure}

We carried out numerical simulations for the KM \eqref{eqn:dsys}
 with $\omega=0$ using the DOP853 solver \cite{HNW93}.
We mainly took $n=500$ but used $n=2000$ for some cases.
The initial values $u_i^n(0)$, $i\in[n]$,
 were independently randomly chosen
 on $[-\pi,\pi]$ or on $[-10^{-3},10^{-3}]$,
 depending on whether completely synchronized states in the CL \eqref{eqn:csys}
 are predicted
 by Theorem~\ref{thm:4a} to be stable or not, i.e., $K$ is smaller or larger than $K_\ell$.
So if the latter initial condition is chosen
 and the completely synchronized state is asymptotically stable,
 then the response of \eqref{eqn:dsys} converges to it as $t\to\infty$.

Figures~\ref{fig:5b1}, \ref{fig:5b2} and \ref{fig:5b3}
 show the time-histories of every $50$th node (from 25th to 475th)
 for case~(i), (ii) and (iii), respectively.
We chose $\kappa=1/3$ in Figs.~\ref{fig:5b1}-\ref{fig:5b3}(a) and (c)
 and $\kappa=1/8$ in Figs.~\ref{fig:5b1}-\ref{fig:5b3}(b),
 while $K=0.65$ or $0.6$ (resp. $0.54$) in Fig.~\ref{fig:5b1} (resp. Fig.~\ref{fig:5b3}),
 and $K=0.325$ or $0.3$ in Fig.~\ref{fig:5b2}.
The values of $K$ are larger or smaller only than $K_\ell$ with $\ell=2$ (resp. with $\ell=6$),
 which is the lowest of $K_j$, $j\in\Nset$, for $\kappa=1/3$ (resp. for $\kappa=1/8$),
 in Figs.~\ref{fig:5b1}-\ref{fig:5b3}(a) and (c) (resp.  in Figs.~\ref{fig:5b1}-\ref{fig:5b3}(b)).
See Eq.~\eqref{eqn:Kl} and Fig.~\ref{fig:3a}.
We observe that the responses converge to the completely synchronized states
 in Figs.~\ref{fig:5b1}-\ref{fig:5b3}(c)
 but to different synchronized states in Figs.~\ref{fig:5b1}-\ref{fig:5b3}(a) and (b),
 as predicted by Theorem~\ref{thm:4a}.
This suggests that the KM \eqref{eqn:dsys} suffers a bifurcation
 as stated in Remark~\ref{rmk:4a}(iii).

\begin{figure}[t]
\begin{minipage}[t]{0.495\textwidth}
\begin{center}
\includegraphics[scale=0.3]{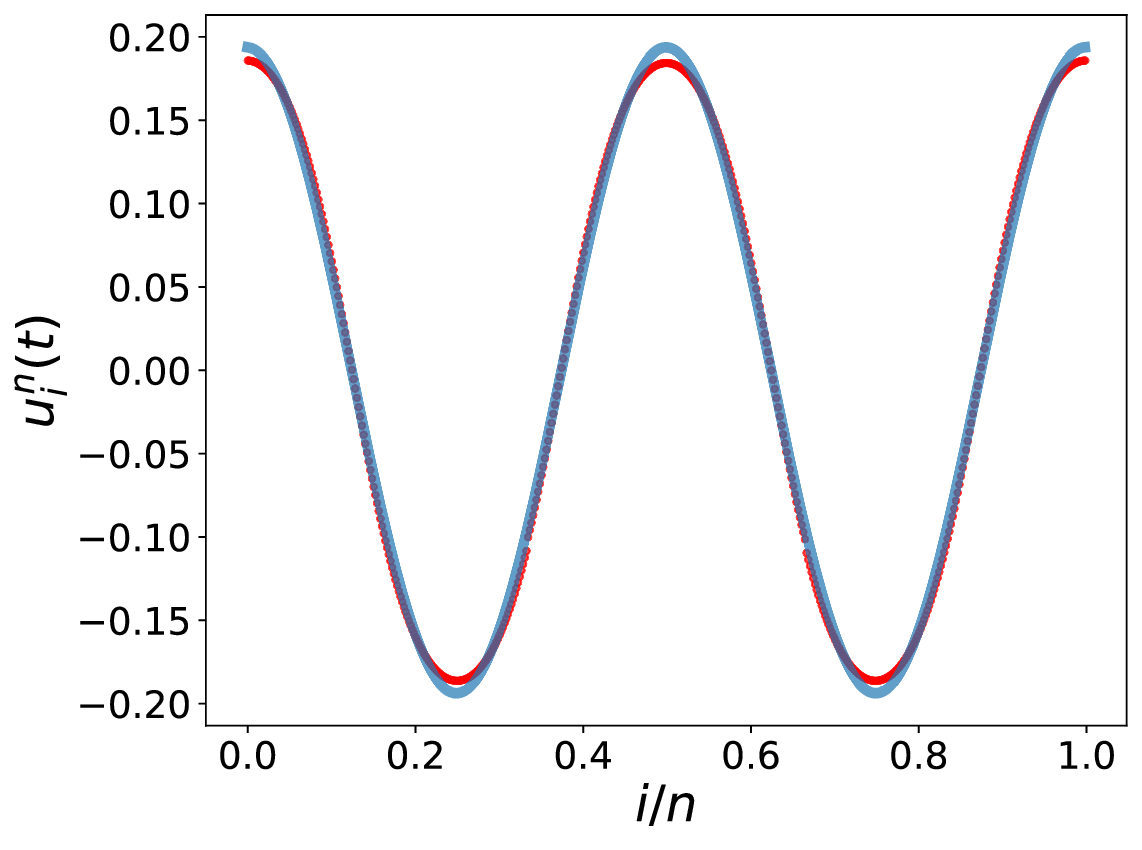}\\
{\footnotesize(a)}
\end{center}
\end{minipage}
\begin{minipage}[t]{0.495\textwidth}
\begin{center}
\includegraphics[scale=0.3]{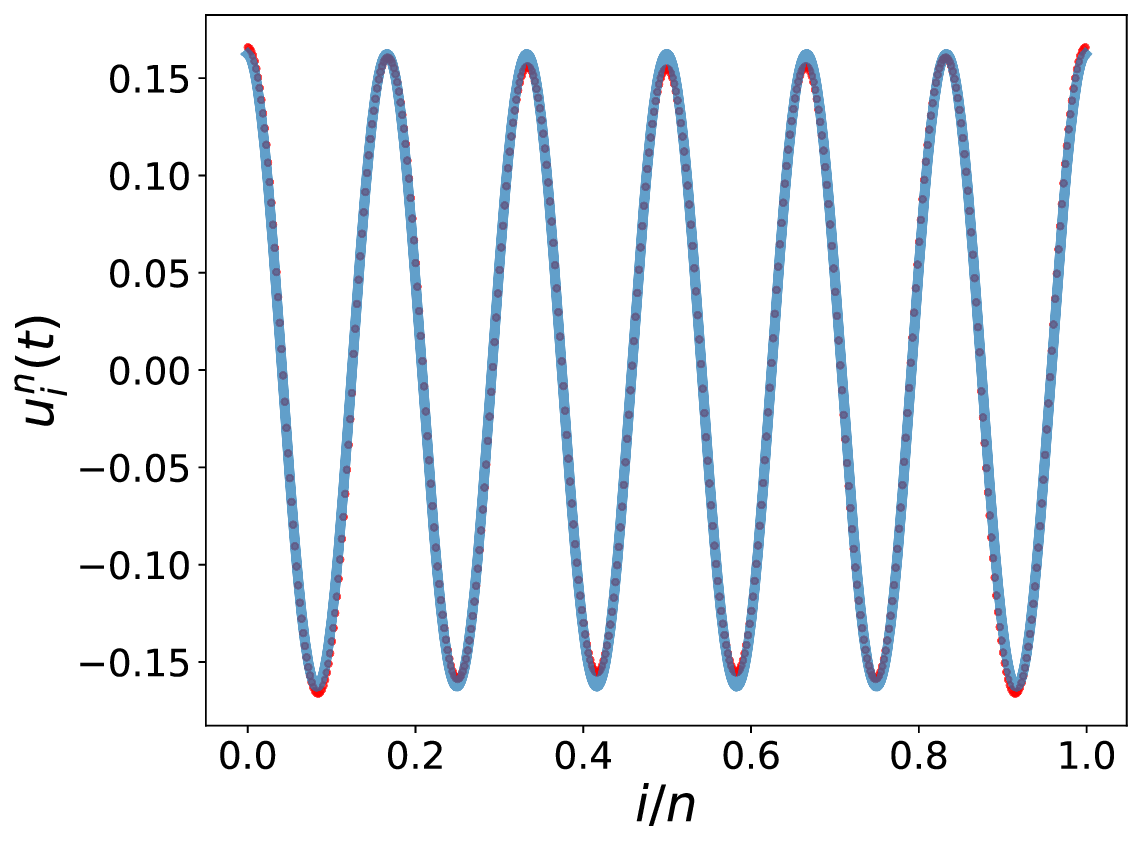}\\
{\footnotesize(b)}
\end{center}
\end{minipage}
\caption{Steady states of the KM \eqref{eqn:dsys}
 with $\omega=0$ and $n=500$ in case~(i):
(a) $(\kappa,K)=(1/3,0.65)$; 
(b) $(1/8,1.7)$.
Here $u_i^n(t)$, $i\in[n]$,  with $t=5000$ or $t=10000$ are plotted as red dots.
 in plates~(a) and (b), respectively.
The blue line represents the curve given by \eqref{eqn:ss}
 fitting the numerical computations. 
\label{fig:5c1}}
\end{figure}

\begin{figure}[t]
\begin{minipage}[t]{0.495\textwidth}
\begin{center}
\includegraphics[scale=0.3]{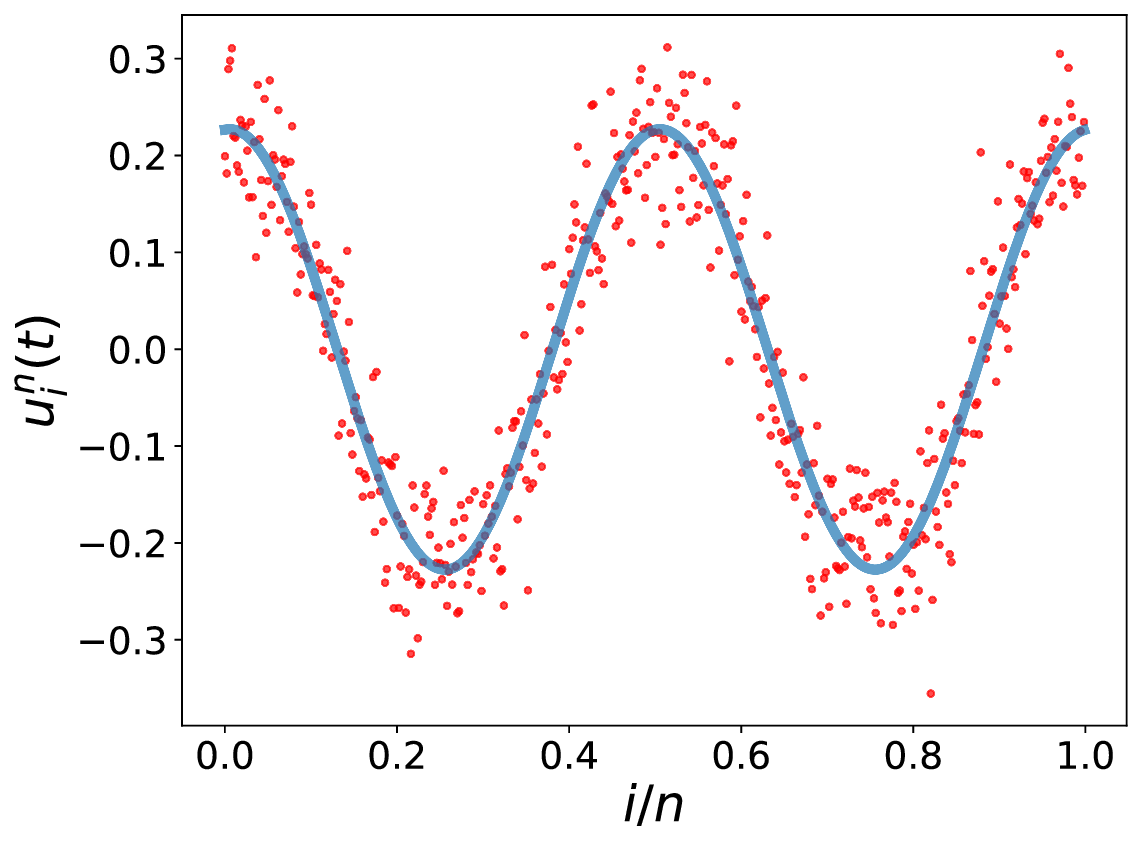}\\
{\footnotesize(a)}
\end{center}
\end{minipage}
\begin{minipage}[t]{0.495\textwidth}
\begin{center}
\includegraphics[scale=0.3]{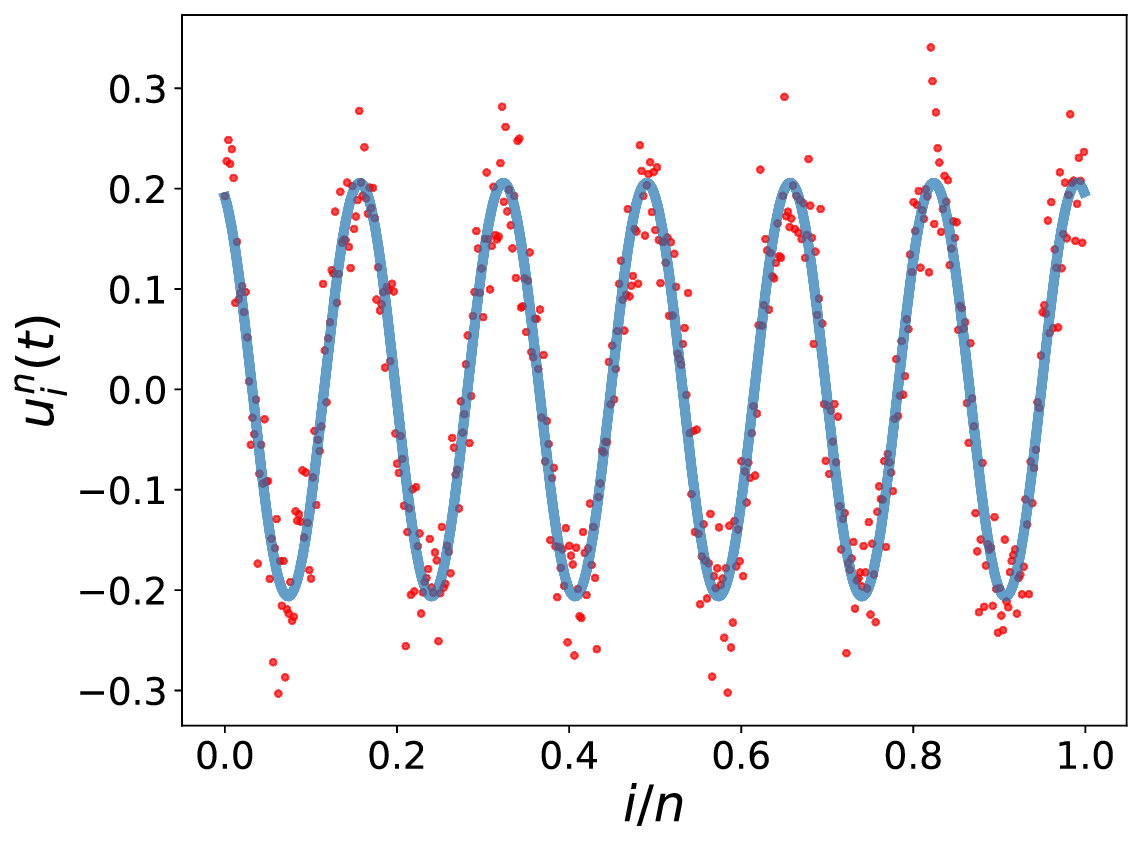}\\
{\footnotesize(b)}
\end{center}
\end{minipage}
\caption{Steady states of the KM \eqref{eqn:dsys}
 with $\omega=0$, $n=500$, $p=0.5$ and $t=5000$ in case~(ii):
(a) $(\kappa,K)=(1/3,0.325)$;  (b) $(1/8,0.85)$.
See also the caption of Fig.~\ref{fig:5c1}.
\label{fig:5c2}}
\end{figure}

\begin{figure}[t]
\begin{minipage}[t]{0.495\textwidth}
\begin{center}
\includegraphics[scale=0.3]{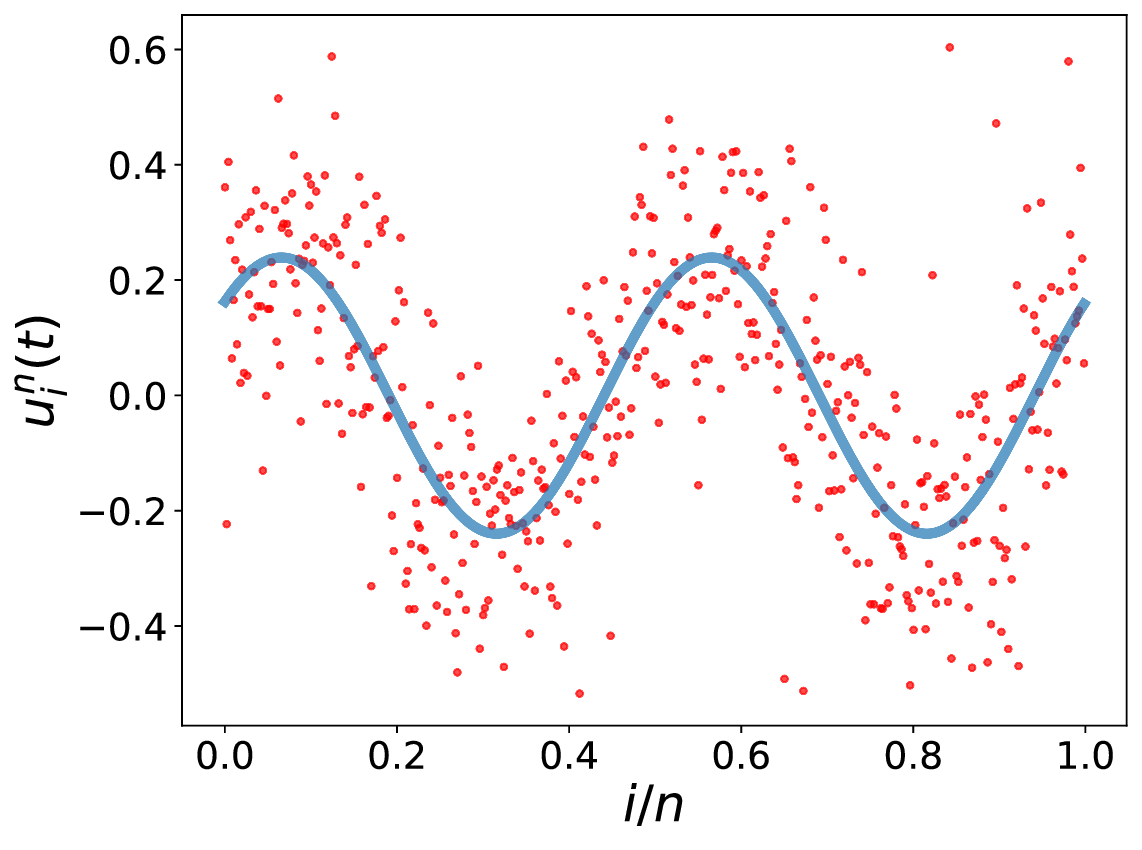}\\
{\footnotesize(a)}
\end{center}
\end{minipage}
\begin{minipage}[t]{0.495\textwidth}
\begin{center}
\includegraphics[scale=0.3]{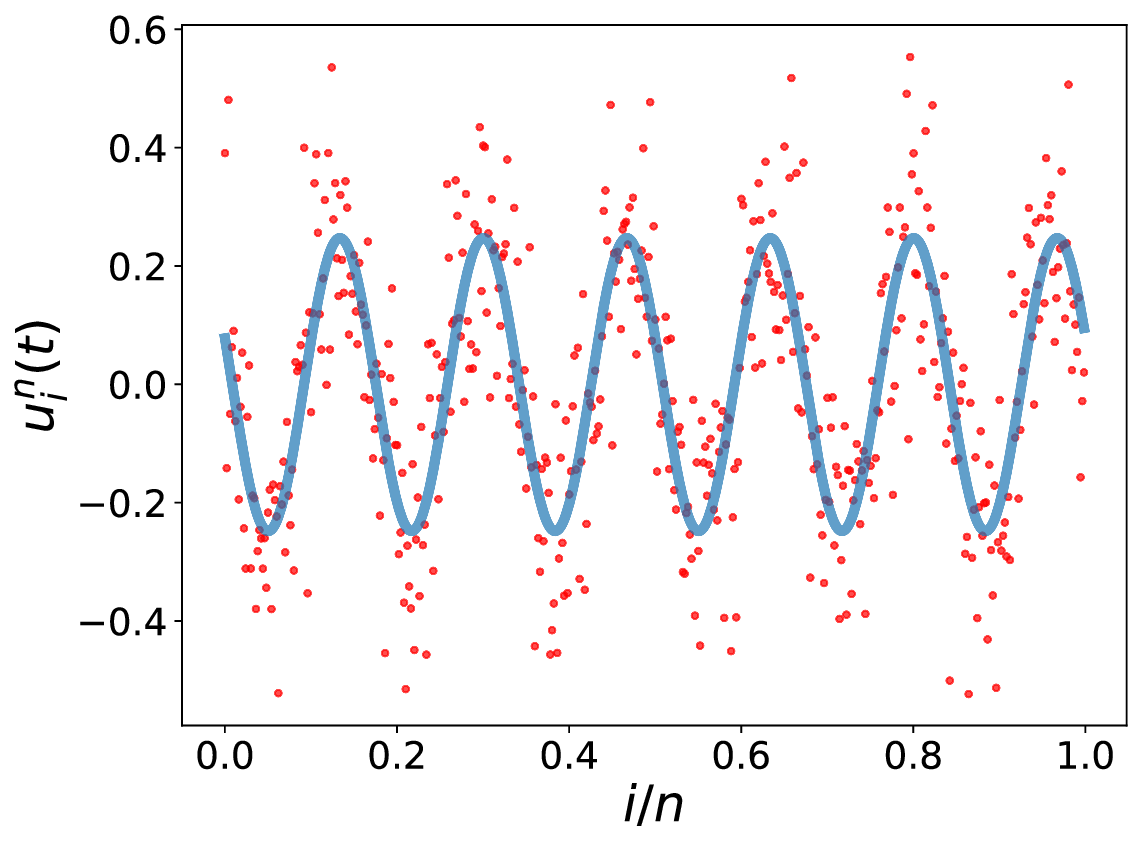}\\
{\footnotesize(b)}
\end{center}
\end{minipage}
\caption{Steady states of the KM \eqref{eqn:dsys}
 with $\omega=0$, $n=500$, $p=1$, $\gamma=0.3$ and $t=5000$ in case~(iii):
(a) $(\kappa,K)=(1/3,0.65)$;  (b) $(1/8,1.7)$.
See also the caption of Fig.~\ref{fig:5c1}.
\label{fig:5c3}}
\end{figure}

In Figs.~\ref{fig:5c1}, \ref{fig:5c2} and \ref{fig:5c3},
 $u_i^n(t)$, $i\in[n]$, are plotted for cases~(i), (ii) and (iii), respectively,
 where the time $t$ was chosen such that they may be regarded as the steady states
 from the results of Figs.~\ref{fig:5b1}-\ref{fig:5b3}.
Here the same values of $K$, $\kappa$, $p$ and $\gamma$
 as in Figs.~\ref{fig:5b1}-\ref{fig:5b3}(a) and (b) were used.
The blue line in each figure represents the most appropriate leading term
 of the stationary solution given in Theorem~\ref{thm:4a}(ii),
\begin{equation}
u(x)=r\sin(2\pi\ell+\psi)+\theta,
\label{eqn:ss}
\end{equation}
with $\ell=2$ or $6$ for the numerical computations,
 where it was estimated by using the least mean square method.
See Table~\ref{tbl:5a} for the estimated values of $r$ and $\psi$ for the line in each figure.
On the other hand, $\theta$ was estimated as $|\theta|<2.5\times 10^{-5}$ in all the figures.
We observe that the agreement between the numerical results
 and theoretical predictions by Theorem~\ref{thm:4a} for the CL \eqref{eqn:csys}
 is fine for case~(i) in Fig.~\ref{fig:5c1} and good for case~(ii) in Fig.~\ref{fig:5c2} 
 although some fluctuations due to randomness are found in the latter.
In particular, they almost completely coincide in Fig~\ref{fig:5c1}.
Thus, the KM \eqref{eqn:dsys} also suffers a bifurcation
 similar to one detected by Theorem~\ref{thm:4a} for  the CL \eqref{eqn:csys},
 as stated in Remark~\ref{rmk:4a}(iii).
However, it is not good for case~(iii) in Fig.~\ref{fig:5c3}
 although the numerical computed steady sate is obviously different
 from the completely synchronized states.
The reason for this result is considered to be that
 the node number $n=500$ is not enough
 for approximation of the KM \eqref{eqn:dsys} by the CL \eqref{eqn:csys} in case~(iii).
Such an observation for the difference between a random sparse graph
 and a deterministic dense or random dense one was also given
 for the classical KM in \cite{IY23},
 although the difference was not so big.

\begin{table}[t]
\caption{Values of $r$ and $\psi$ for the most appropriate leading term \eqref{eqn:ss}
 of the stationary solution in Figs.~\ref{fig:5c1}-\ref{fig:5c3}.
Here they are given with five decimal places.
\label{tbl:5a}}
\begin{tabular}{c|c|c|c|c|c|c}
\hline
& Fig.~\ref{fig:5c1}(a) & Fig.~\ref{fig:5c1}(b)  & Fig.~\ref{fig:5c2}(a) & Fig.~\ref{fig:5c2}(b) &
 Fig.~\ref{fig:5c3}(a) & Fig.~\ref{fig:5c3}(b)\\
\hline
 $r$ & $0.19379$ & $0.16242$ & $0.22731$ & $0.20599$ & $0.23947$ & $0.24756$\\
 $\psi$ & $1.58693$ & $1.60601$ & $1.50260$ & $1.94671$ & $0.74919$ & $2.82065$\\
\hline
\end{tabular}
\end{table}

\begin{figure}[t]
\begin{minipage}[t]{0.495\textwidth}
\begin{center}
\includegraphics[scale=0.3]{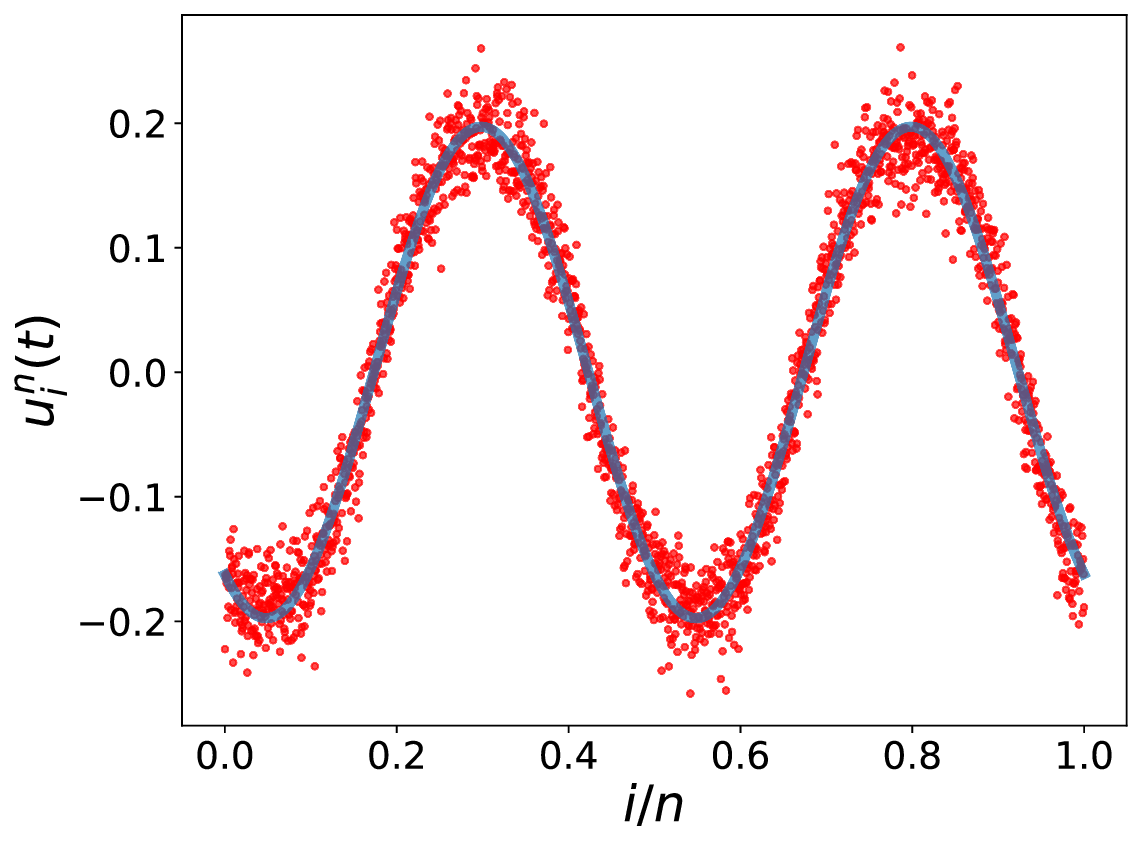}\\
{\footnotesize(a)}
\end{center}
\end{minipage}
\begin{minipage}[t]{0.495\textwidth}
\begin{center}
\includegraphics[scale=0.3]{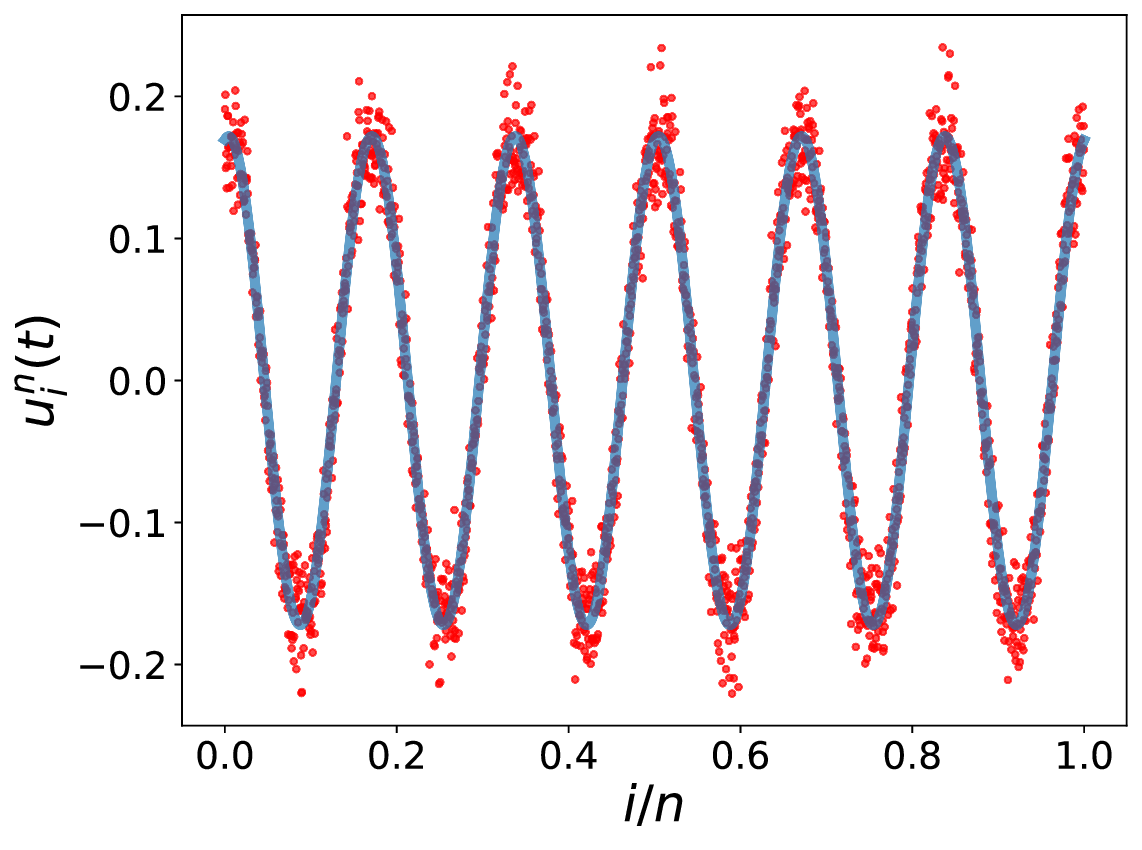}\\
{\footnotesize(b)}
\end{center}
\end{minipage}
\caption{Steady states of the KM \eqref{eqn:dsys}
 with $\omega=0$, $n=2000$, $p=0.5$ and $t=1000$ in case~(ii):
(a) $(\kappa,K)=(1/3,0.325)$;  (b) $(1/8,0.85)$.
See also the caption of Fig.~\ref{fig:5c1}.
\label{fig:5d1}}
\end{figure}

\begin{figure}[t]
\begin{minipage}[t]{0.495\textwidth}
\begin{center}
\includegraphics[scale=0.3]{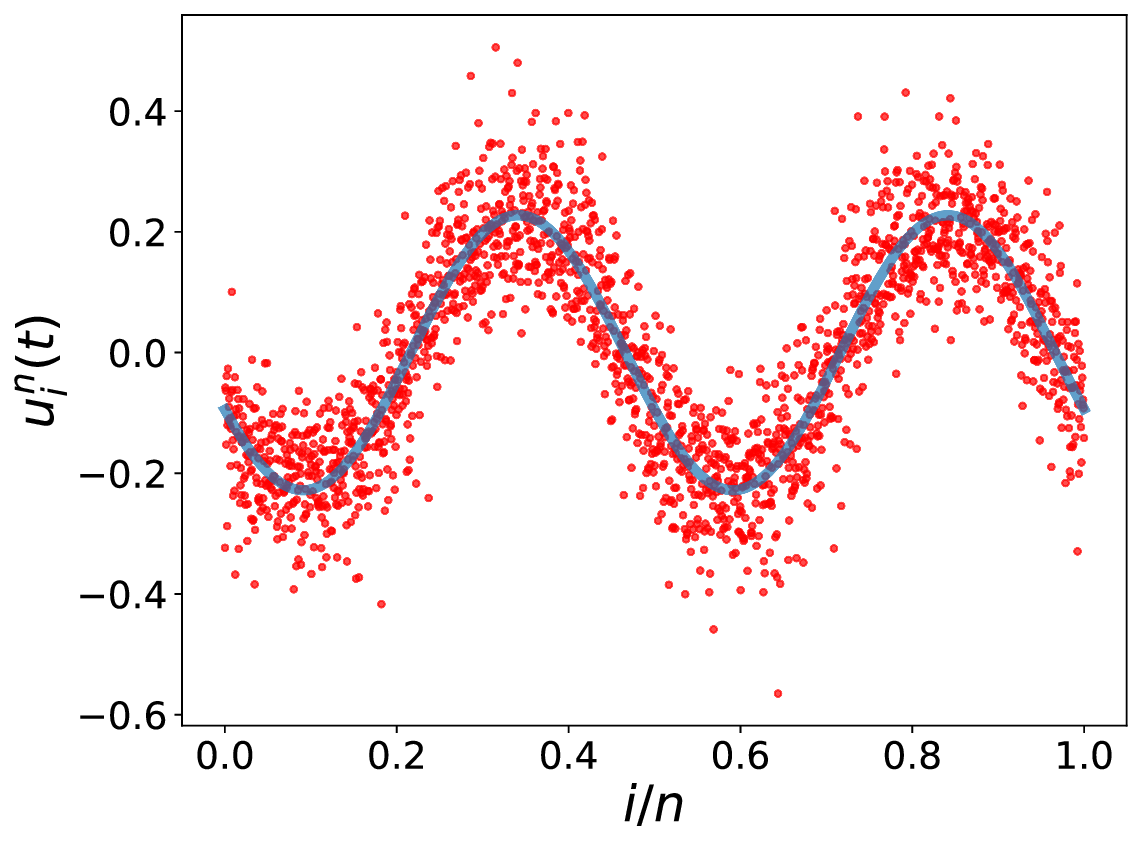}\\
{\footnotesize(a)}
\end{center}
\end{minipage}
\begin{minipage}[t]{0.495\textwidth}
\begin{center}
\includegraphics[scale=0.3]{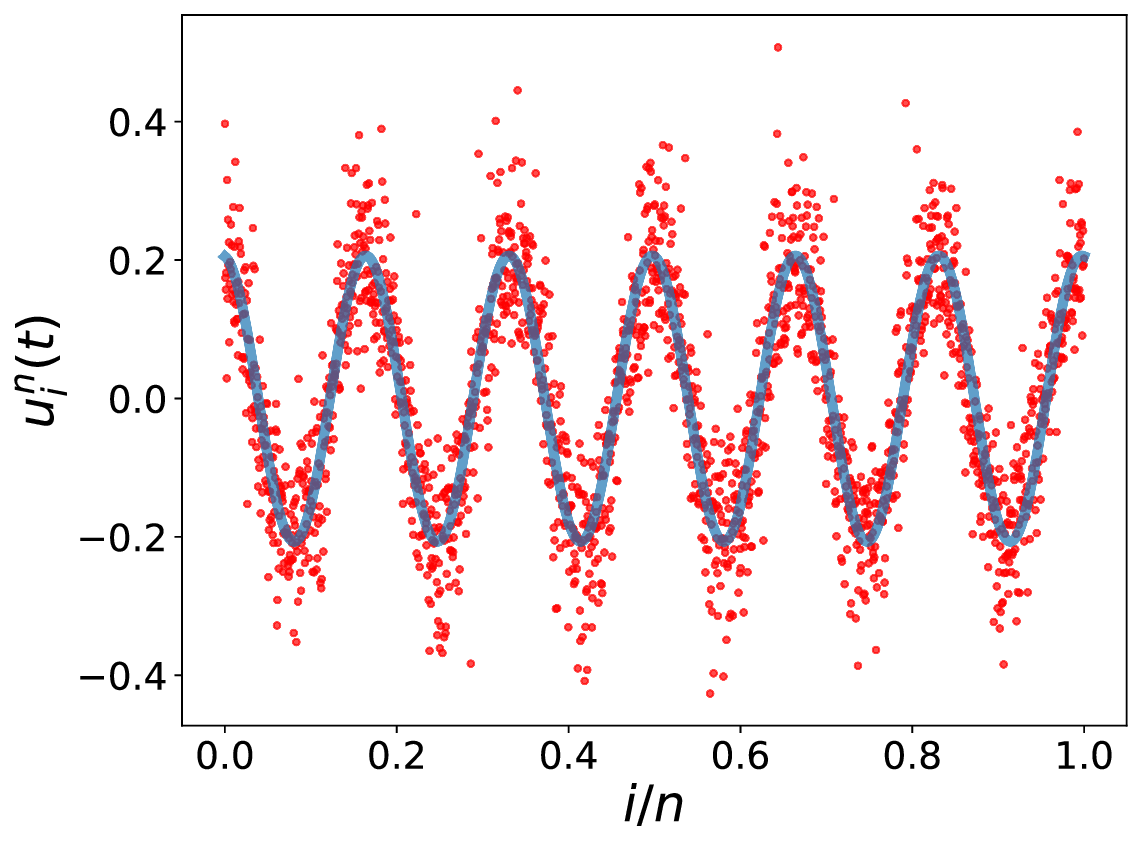}\\
{\footnotesize(b)}
\end{center}
\end{minipage}
\caption{Steady states of the KM \eqref{eqn:dsys}
 with $\omega=0$, $n=2000$, $p=1$, $\gamma=0.3$ and $t=1000$ in case~(iii):
(a) $(\kappa,K)=(1/3,0.65)$;  (b) $(1/8,1.7)$.
See also the caption of Fig.~\ref{fig:5c1}.
\label{fig:5d2}}
\end{figure}

Figures~\ref{fig:5d1} and \ref{fig:5d2}
 show the same results as Figs.~\ref{fig:5c2} and \ref{fig:5c3}, respectively,
 not for $n=500$ but for $n=2000$.
The approximations of the KM \eqref{eqn:dsys} 
 by the CL \eqref{eqn:csys} are better than Figs.~\ref{fig:5c2} and \ref{fig:5c3}.
See Table~\ref{tbl:5b} for the values of $r$ and $\psi$ in \eqref{eqn:ss}
 for the numerical computations.
On the other hand, $\theta$ was estimated as $|\theta|<2\times 10^{-5}$
 in these figures.

\begin{table}[t]
\caption{Values of $r$ and $\psi$ for the most appropriate leading term \eqref{eqn:ss}
 of the stationary solution in Figs.~\ref{fig:5d1} and \ref{fig:5d2}.
Here they are given with five decimal places.
\label{tbl:5b}}
\begin{tabular}{c|c|c|c|c}
\hline
& Fig.~\ref{fig:5d1}(a) & Fig.~\ref{fig:5d1}(b) & Fig.~\ref{fig:5d2}(a) & Fig.~\ref{fig:5d2}(b)\\
\hline
 $r$ & $0.19726$ & $0.17194$ & $0.22777$ & $0.20647$\\
 $\psi$ & $-2.17431$ & $ 1.40843$ & $-2.71436$ & $1.67886$\\
\hline
\end{tabular}
\end{table}

\begin{figure}[t]
\begin{minipage}[t]{0.495\textwidth}
\begin{center}
\includegraphics[scale=0.495]{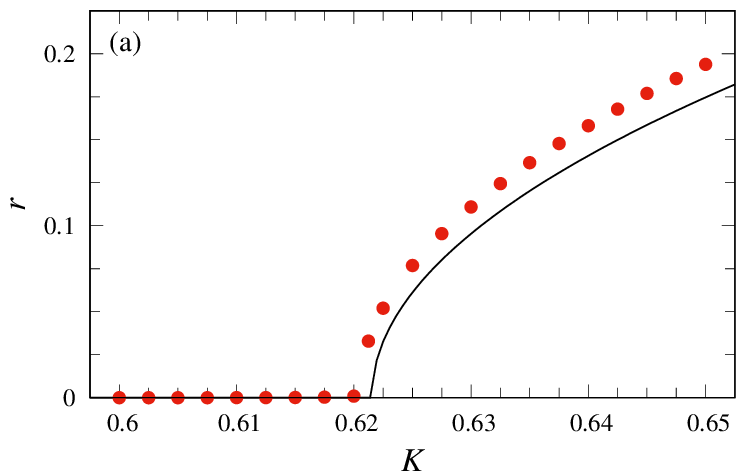}
\end{center}
\end{minipage}
\begin{minipage}[t]{0.495\textwidth}
\begin{center}
\includegraphics[scale=0.495]{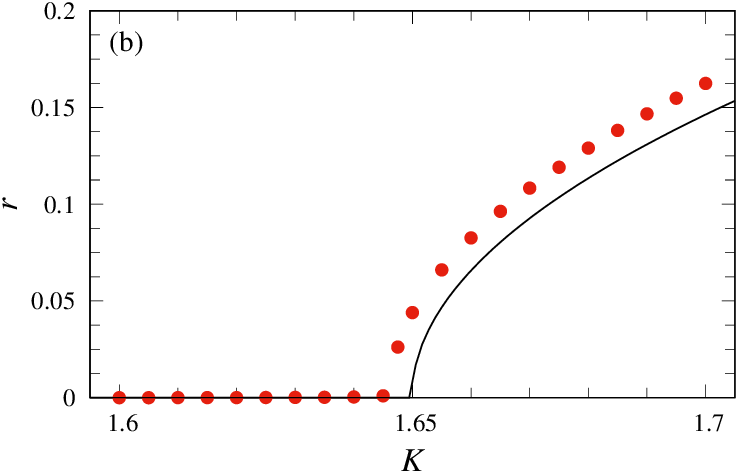}
\end{center}
\end{minipage}
\caption{Numerically computed bifurcation diagrams for the KM \eqref{eqn:dsys}
 with $\omega=0$ and $n=500$ in case~(i):
(a) $\kappa=1/3$;  (b) $\kappa=1/8$.
The ordinate represents the amplitude $r$
 of the $\ell$-humped sinusoidal shape \eqref{eqn:ss}
 estimated for the numerically computed steady state.
The small red disks and black lines represent numerical computations
 and theoretical predictions for the CL \eqref{eqn:csys} in Theorem~\ref{thm:4a}.
\label{fig:5e}}
\end{figure}

In Fig.~\ref{fig:5e}
 we give numerically computed bifurcation diagrams for the KM \eqref{eqn:dsys}
 with $n=500$ in case~(i).
Here the same value of $\kappa$
 as in Figs.~\ref{fig:5b1} and \ref{fig:5c1} were used.
The ordinate represents the amplitude $r$
 of the $\ell$-humped sinusoidal shape \eqref{eqn:ss}
 estimated for the numerically computed  steady state,
 as in Figs.~\ref{fig:5c1}-\ref{fig:5d2} and Tables~\ref{tbl:5a} and \ref{tbl:5b}.
The black line represents the theoretical predictions for the CL \eqref{eqn:csys}
 in Theorem~\ref{thm:4a}.
Their agreement is qualitatively good
 although differences due to finite size effects of the KM \eqref{eqn:dsys} are found.
 
\section*{Acknowledgements}
This work was partially supported by the JSPS KAKENHI Grant Number JP23K22409.


\appendix

\renewcommand{\theequation}{\Alph{section}.\arabic{equation}}

\section{Derivation of \eqref{eqn:ifex}}

We first rewrite the CL \eqref{eqn:csys} as
\begin{align}
&
\frac{\partial}{\partial t}u(t,x)\notag\\
&
=p\cos u(t,x)\int_I\sin u(t,y)\d y-K\cos 2u(t,x)\int_I W_2(x,y)\sin 2u(t,y)\d y\notag\\
&\quad
 -p\sin u(t,x)\int_I\cos u(t,y)\d y+K\sin 2u(t,x)\int_I W_2(x,y)\cos2u(t,y)\d y.
\label{eqn:csys1}
\end{align}
Letting \eqref{eqn:solex}, we have
\begin{align*}
&\cos u(t,y)\\
&=
1-\tfrac{1}{4}((\xi_\ell^2+\eta_\ell^2)
+(\xi_\ell^2-\eta_\ell^2)\cos 4\pi\ell y
 +2\xi_\ell\eta_\ell\sin 4\pi\ell y))+O(4),\\[2ex]
&
\sin u(t,y)\\
&=\sum_{i=1}^\infty(\xi_i\cos 2\pi iy+\eta_i\sin2\pi iy)
 -\tfrac{1}{8}(\xi_\ell^2+\eta_\ell^2)(\xi_\ell\cos2\pi\ell y+\eta_\ell\sin 2\pi\ell y)\\
&\quad
-\tfrac{1}{24}((\xi_\ell^2-3\eta_\ell^2)\xi_\ell\cos 6\pi\ell y
 +(3\xi_\ell^2-\eta_\ell^2)\eta_\ell\sin 6\pi\ell y)
 +O(5),
\end{align*}
and compute
\begin{align*}
\int_I\cos u(t,x)\d y=1-\tfrac{1}{4}(\xi_\ell^2+\eta_\ell^2)+O(4),\quad
\int_I\sin u(t,x)\d y=O(5)
\end{align*}
and
\begin{align*}
&
\int_I W_2(x,y)\cos 2u(t,y)\d y\\
&
=2\kappa-2\kappa(\xi_\ell^2+\eta_\ell^2)
 -\delta_{2\ell}((\xi_\ell^2-\eta_\ell^2)\cos 4\pi\ell x
 +2\xi_\ell\eta_\ell\sin 4\pi\ell x)+O(4),\\
&
\int_I W_2(x,y)\sin 2u(t,y)\d y\\
&=2\sum_{i=1}^\infty\delta_i(\xi_i\cos 2\pi ix+\eta_i\sin2\pi ix)\\
&\quad
-\delta_\ell(\xi_\ell^2+\eta_\ell^2)(\xi_\ell\cos2\pi\ell x+\eta_\ell\sin 2\pi\ell x)\\
&\quad
-\tfrac{1}{3}\delta_{3\ell}((\xi_\ell-3\eta_\ell^2)\xi_\ell\cos 6\pi\ell x
 +(3\xi_\ell-\eta_\ell^2)\eta_\ell\sin 6\pi\ell x)
 +O(5),\\
\end{align*}
where $O(k)$ represents a higher-order term of
\[
O\left((\xi_\ell^2+\eta_\ell^2)^{k/2}+\sum_{j=1,j\neq\ell}^\infty( \xi_i^2+\eta_i^2)^{3/2}\right).
\]
We substitute \eqref{eqn:solex} into \eqref{eqn:csys1}
 and integrate the resulting equation with respect to $x$ from $0$ to $1$
 after multiplying it with $\cos 2\pi j$ or $\sin 2\pi j$, $j\in\Nset$.
Note that the right-hand side becomes the Fourier coefficient of an $L^2(I)$-function 
 and by Theorem~\ref{thm:2a} $\dot{\mathbf{u}}(t)\in L^2(I)$.
Thus, we obtain \eqref{eqn:ifex} after lengthy calculations.

\section{Validity of Application of the Center Manifold Theory}

In this appendix
 we briefly explain the validity of application of the center manifold theory \cite{HI11}
 in Section~4.
See Chapter~2 of \cite{HI11} for the details of the theory.
A more general case is treated there.

Consider a general dynamical system
\begin{equation}
\frac{\d\mathbf{u}}{\d t}=\L\mathbf{u}+R(\mathbf{u};\mu)
\label{eqn:b1}
\end{equation}
in $L^2(I)$, where $\L:L^2(I)\to L^2(I)$ is a linear bounded operator like \eqref{eqn:lep}
 and $R\in C^l(L^2(I)\times\Rset,L^2(I))$ ($l\ge 2$)
 satisfies $R(\mathbf{0},0)=\mathbf{0}$ and $\D_\mathbf{u}R(\mathbf{0},0)=\mathbf{0}$
 when $\mu=0$.
In particular, the origin $\mathbf{u}=\mathbf{0}$ is a solution to \eqref{eqn:b1}.
The CL \eqref{eqn:csys} has the form \eqref{eqn:b1}.

Let the spectrum of $\L$ be decomposed as
\[
\sigma(\L)=\sigma_+(\L)\cup\sigma_-(\L)\cup\sigma_0(\L),
\] 
where $\sigma_+(\L)$, $\sigma_-(\L)$ and $\sigma_0(\L)$, respectively,
 are the unstable, stable and center spectrum,
 and consist of complex numbers with positive, negative and zero real parts.
We assume the following:
\begin{enumerate}
\setlength{\leftskip}{-0.8em}
\item[\bf(H1)]
There exists a constant $c>0$ such that
\[
\inf_{\lambda\in\sigma_+}\lambda>c,\quad
\sup_{\lambda\in\sigma_-}\lambda<-c;
\]
Moreover, the set $\sigma_0$ consists
 of a finite number of eigenvalues with finite algebraic multiplicities.
\end{enumerate}
It follows from the result of Section~3
 that the CL \eqref{eqn:csys} satisfies (H1)
 at $K=K_\ell$ under the hypothesis of Theorem~\ref{thm:4a}.

Define the spectral projection by the Dunford integral formula as
\[
\P_0=\frac{1}{2\pi i}\int_\Gamma(\lambda\id-\L)^{-1}\d\lambda
\]
where $\id:L^2(I)\to L^2(I)$ represents the identity operator
 and $\Gamma\subset\{\lambda\in\Cset\mid |\Re\lambda|<c\}$
 is a simple, oriented counterclockwise, Jordan curve surrounding $\sigma_0$.
Let $X_0=\mathrm{Im}\P_0$ and $X_\h=\mathrm{ker}\P_0$,
 and let $\L_0$ and $\L_\h$ denote the restriction
 of $\L$ to $X_0$ and $X_\h$, respectively.
So we have
\[
L^2(I)=X_0\oplus X_\h
\]
and
\[
\sigma(\L_0)=\sigma_0(\L),\quad
\sigma(\L_\h)=\sigma_+(\L)\cup\sigma_-(\L).
\]
Let
\[
C_b(\Rset,L^2(I))=\Bigl\{\mathbf{u}(t)\in C(\Rset,L^2(I))
 \,\Big|\,\sup_{t\in\Rset}\bigl(e^{-b|t|}\|\mathbf{u}(t)\|\bigr)<\infty\Bigr\},
\]
where $b>0$ is a constant.
We also assume the following:
\begin{enumerate}
\setlength{\leftskip}{-0.8em}
\item[\bf(H2)]
For any $b\in[0,c]$ and any $\mathbf{f}\in C_b(\Rset,L^2(I))$
 the linear system
\begin{equation}
\frac{\d\mathbf{u}}{\d t}=\L_\h\mathbf{u}+\mathbf{f}(t)
\label{eqn:b2}
\end{equation}
has a unique solution $\mathbf{u}=:\K_\h\mathbf{f}\in C_b(\Rset,L^2(I))$.
Furthermore, the linear operator $\K_\h:C_b(\Rset,L^2(I))\to C_b(\Rset,L^2(I))$ is bounded
 and there exists a continuous map $C:[0,c]\to\Rset$ such that
\[
\|\K_\h\|\le C(b),
\]
where the norm $\|\cdot\|$ represents the operator norm.
\end{enumerate}
Since the Fourier expansion of any function in $L^2(I)$ converges to itself a.e.
 by Carleson's theorem \cite{C66},
we write \eqref{eqn:b2} as
\begin{equation}
\dot{\xi}_j=-\beta_j\xi_j+\hat{f}_{1j}(t),\quad
\dot{\eta}_j=-\beta_j\eta_j+\hat{f}_{2j}(t),\quad
j\neq\ell,
\label{eqn:b3}
\end{equation}
for the CL \eqref{eqn:csys} at $K=K_\ell$
 under the restriction \eqref{eqn:con} like \eqref{eqn:ifex},
 where $\hat{f}_{1j}(t)$ and $\hat{f}_{2j}(t)$ are the Fourier coefficients of $\mathbf{f}(t)$,
\[
f(t,x)=\sum_{i\neq\ell}(\hat{f}_{1i}(t)\cos 2\pi ix+\hat{f}_{2i}(t)\sin 2\pi ix).
\]
The linear system \eqref{eqn:b3} has a unique solution
\begin{equation}
\xi_j(t)=\int_{-\infty}^t e^{-\beta_j(t-\tau)}\hat{f}_{1j}(\tau)\d\tau,\quad
\eta_j(t)=\int_{-\infty}^t e^{-\beta_j(t-\tau)}\hat{f}_{1j}(\tau)\d\tau,\quad
j\neq\ell,
\label{eqn:b4}
\end{equation}
that belongs to $C_b(\Rset,L^2(I))$ if $\beta_j>0$ for any $j\neq\ell$.
Similarly we can easily show that the system \eqref{eqn:b3} has a unique solution
 in $C_b(\Rset,L^2(I))$ even if $\beta_j<0$ for some $j\neq\ell$
 (cf. Exercise 2.8 in Chapter~2 of \cite{HI11}).
 It is easy to show that the solution \eqref{eqn:b4} satisfies the remaining part of (H2).
Thus, the CL \eqref{eqn:csys} also satisfies (H2)
 at $K=K_\ell$ under the hypothesis of Theorem~\ref{thm:4a}.
Using Theorem~3.3 of Chapter~2 of \cite{HI11},
 we obtain the following.
  
\begin{thm}
\label{thm:b1}
Suppose that hypotheses {\rm(H1)} and {\rm(H2}) hold
 and let $O_u$, $O_0$ and $O_\mu$ be neighborhoods of the origins
 in $L^2(I)$, $X_0$ and $\Rset$, respectively.
Then there exists a map $\Psi\in C^l(O_0\times O_\mu,X_\h)$  such that
\[
\Psi(\mathbf{0},0)=\mathbf{0},\quad
\D\Psi(\mathbf{0},0)=\mathbf{0}
\]
and the manifold
\[
W^\c
 =\{\mathbf{u}_0+\Psi(\mathbf{u}_0,\mu)\mid(\mathbf{u}_0,\mu)\in O_0\times O_\mu\}
 \subset O_u\times O_\mu
\]
satisfies the following properties$:$
\begin{enumerate}
\setlength{\leftskip}{-1.8em}
\item[(i)]
$W^\c$ is locally invariant for $\mu\in O_\mu$,
 i.e., if $\mathbf{u}(t)$ is a solution to \eqref{eqn:b1}
 satisfying $\mathbf{u}(0)\in W^\c$ and $\mathbf{u}(t)\in O_u$ on $[0,\tau]$
 for some $\tau>0$, then  $\mathbf{u}(t)\in W^\c$ on $[0,\tau];$
\item[(ii)]
$W^\c$ contains all bounded solutions to \eqref{eqn:b1}
 staying in $O_u$ for all $t\in\Rset$.
\end{enumerate}
This means that $W^\c$ is a center  manifold.
\end{thm}

Thus, there exists the center manifold $W^\c$ given by \eqref{eqn:Wc}
 in the CL \eqref{eqn:csys} under the condition~\eqref{eqn:con}.



\begin{thebibliography}{99}
\bibitem{ABVRS05}
J.A.~Acebr\'on, L.L.~Bonilla, C.J.P.~Vicente, F.~Ritort and R.~Spigler,
The Kuramoto model: A simple paradigm for synchronization phenomena,
\textit{Rev. Mod. Phys.}, \textbf{77} (2005), 137--185.

\bibitem{ADKMZ08}
A.~Arenas, A.~Diaz-Guilera, J.~Kurths, Y.~Moreno and C.~Zhou,
Synchronization in complex networks,
\textit{Phys. Rep.},\textbf{469}(2008), 93--153.

\bibitem{C66}
 L.~Carleson, 
On convergence and growth of partial sums of  Fourier series,
\textit{Acta Mathematica}, \textbf{116} (1966), 135--157.

\bibitem{C15}
H.~Chiba,
A proof of the Kuramoto conjecture for a bifurcation structure
 of the infinite-dimensional Kuramoto model,
\textit{Ergodic Theory Dynam. Systems}, \textbf{35} (2015), 762--834.
\bibitem{CN11}
H.~Chiba and I.~Nishikawa,
Center manifold reduction for large populations of globally coupled phase oscillators,
\textit{Chaos}, \textbf{21} (2011), 043103.

\bibitem{CL55}
E.A.~Coddington and N.~Levinson,
\textit{Theory of Ordinary Differential Equations},
McGraw-Hill, New York, 1955.

\bibitem{CD99}
J.D.~Crawford and K.T.R.~Davies,,
Synchronization of globally coupled phase oscillators:
 singularities and scaling for general couplings,
\textit{Phys. D}, \textbf{125} (1999), 1--46.

\bibitem{DB14}
F.~D\"orfler and F.~Bullo,
Synchronization in complex networks of phase oscillators: A survey,
\textit{Automatica}, \textbf{50} (2014), 1539--1564.

\bibitem{E85}
G.B.~Ermentrout,
Synchronization in a pool of mutually coupled oscillators with random frequencies
\textit{J. Math. Biol.}, \textbf{22} (1985), 1--9.

\bibitem{GC20}
S.~Gao and P.E.~Caines,
Graphon control of large-scale networks of linear systems,
\textit{IEEE Trans. Automat. Contr.}, \textbf{65} (2020), 4090--4105.
\bibitem{GC21}
S.~Gao and P.E.~Caines,
Subspace decomposition for graphon LQR:
Applications to VLSNs of harmonic oscillators,
\textit{IEEE Trans. Control. Netw. Syst.}, \textbf{8} (2021), 576--586.

\bibitem{GCH23}
S.~Gao, P.E.~Caines and M.Y.~Huang,
LQG graphon mean field games: Analysis via graphon-invariant subspaces,
\textit{IEEE Trans. Automat. Contr.}, \textbf{68} (2023), 7482--7497.
\bibitem{GM22}
S.~Gao and A.~Mahajan,
Optimal control of network-coupled subsystems:
 Spectral decomposition and low-dimensional solutions,
\textit{IEEE Trans. Control. Netw. Syst.}, \textbf{9} (2022), 657--669.

\bibitem{GHM12}
T.~Girnyk, M.~Hasler and Y.~Maistrenko,
Multistability of twisted states in non-locally coupled Kuramoto-type models,
\textit{Chaos}, \textbf{22} (2012), 013114.

\bibitem{GH83}
J.~Guckenheimer and P.~Holmes,
\textit{Nonlinear Oscillations, Dynamical Systems, and Bifurcations of Vector Fields},
Springer, New York, 1983.

\bibitem{HNW93}
E.~Hairer, S.P.~N{\o}rsett and G.~Wanner,
\textit{Solving Ordinary Differential Equations I: Nonstiff Problems}, 2nd ed.
Springer, Berlin, 1993.

\bibitem{HI11}
M.~Haragus and G.~Iooss,
\textit{Local Bifurcations, Center Manifolds, and Normal Forms
 in Infinite-Dimensional Dynamical Systems},
Springer, London, 2011.

\bibitem{IY23}
R.~Ihara and K.~Yagasaki,
Continuum limits of coupled oscillator networks depending on multiple sparse graphs,
\textit{J. Nonlinear Sci.}, \textbf{33} (2023), 62;
Correction, \textbf{35} (2025), 27.

\bibitem{KM17} 
D. Kaliuzhnyi-Verbovetskyi and G. S. Medvedev, 
The semilinear heat equation on sparse random graphs, 
\textit{SIAM J. Math. Anal.}, \textbf{49} (2017), no. 2, 1333-1355. 

\bibitem{KP13}
M.~Komarov and A.~Pikovsky,
Multiplicity of singular synchronous states in the Kuramoto model of coupled oscillators,
\textit{Phys. Rev. Lett.}, \textbf{111} (2013), 204101.

\bibitem{KP14}
M.~Komarov and A.~Pikovsky,
The Kuramoto model of coupled oscillators with a bi-harmonic coupling function,
\textit{Phys. D}, \textbf{289} (2014), 18--31; Erratum, \textbf{313} (2015), 117.

\bibitem{KP15}
M.~Komarov and A.~Pikovsky,
Synchronization transitions in ensembles of noisy oscillators with bi-harmonic coupling,
\textit{J. Phys. A}, \textbf{48} (2015), 105101.

\bibitem{K75} 
Y.~Kuramoto,
Self-entrainment of a population of coupled non-linear oscillators,
in \textit{International Symposium on Mathematical Problems in Theoretical Physics},
H.~Araki (ed.),
Springer, Berlin, 1975, pp. 420--422.
\bibitem{K84} 
Y.~Kuramoto, 
\textit{Chemical Oscillations, Waves, and Turbulence}, Springer, Berlin, 1984. 

\bibitem{K04}
Y.A.~Kuznetsov,
\textit{Elements of Applied Bifurcation Theory},
Springer, New York, 2004.

\bibitem{L12} 
L.~Lov\'asz, 
\textit{Large Networks and Graph Limits}, 
AMS, Providence RI, 2012. 

\bibitem{M14a} 
G.S.~Medvedev, 
The nonlinear heat equation on dense graphs and graph limits, 
\textit{SIAM J. Math. Anal.}, \textbf{46} (2014), 2743--2766. 
\bibitem{M14b} 
G.S.~Medvedev, 
The nonlinear heat equation on W-random graphs, 
\textit{Arch. Ration. Mech. Anal.}, \textbf{212} (2014), 781--803. 
\bibitem{M19} 
G.S.~Medvedev, 
The continuum limit of the Kuramoto model on sparse random graphs, 
\textit{Comm. Math. Sci.}, \textbf{17} (2019), 883--898. 

\bibitem{PR15}
A.~Pikovsky and M.~Rosenblum,
Dynamics of globally coupled oscillators: Progress and perspectives,
\textit{Chaos}, \textbf{25} (2015), 097616.

\bibitem{PRK01}
A. Pikovsky, M. Rosenblum, and J. Kurths,
\textit{Synchronization: A Universal Concept in Nonlinear Sciences},
Cambridge University Press, Cambridge, 2001.

\bibitem{RPJK16}
F.A.~Rodrigues, T.K.DM.~Peron, P.~Ji and J.~Kurths,
The Kuramoto model in complex networks,
\textit{Phys. Rep.}, \textbf{610} (2016), 1--98.


\bibitem{S00}
S.H.~Strogatz,
From Kuramoto to Crawford: Exploring the onset of synchronization
 in populations of coupled oscillators,
\textit{Phys. D}, \textbf{143} (2000), 1--20.

\bibitem{WSG06}
D.A.~Wiley, S.H.~Strogatz and M.~Girvan,
The size of the sync basin,
\textit{Chaos}, \textbf{16} (2006), 015103.

\bibitem{Y24}
K.~Yagasaki,
Bifurcations and stability of synchronized solutions in the Kuramoto model
 with uniformly spaced natural frequencies,
\textit{Nonlinearity}, \textbf{28} (2025), 075032.
\end{thebibliography}
\end{document}